\documentclass{article}
\usepackage{amssymb}
\usepackage{amsfonts}
\usepackage{url}
\usepackage{amsmath}

\newtheorem{theorem}{Theorem}
\newtheorem{acknowledgement}{Acknowledgement}

\newtheorem{corollary}{Corollary}

\newtheorem{lemma}{Lemma}

\newtheorem{proposition}{Proposition}
\newtheorem{remark}{Remark}

\newenvironment{proof}[1][Proof]{\textbf{#1.} }{\ \rule{0.5em}{0.5em}}

\begin{document}

\begin{center}
{\Large Linear Half-Space Problems in Kinetic Theory: Abstract Formulation
and Regime Transitions}\bigskip

{\large Niclas Bernhoff\smallskip }

Department of Mathematics, Karlstad University, 65188 Karlstad, Sweden

niclas.bernhoff@kau.se
\end{center}

\textbf{Abstract:}{\small \  Half-space problems in the kinetic theory of
gases are of great importance in the study of the asymptotic behavior of
solutions of boundary value problems for the Boltzmann equation for small
Knudsen numbers. In this work a generally formulated half-space problem,
based on generalizations of stationary half-space problems in one spatial
variable for the Boltzmann equation - for hard-sphere models of monatomic
single species and multicomponent mixtures - is considered. The number of
conditions on the indata at the interface needed to obtain well-posedness is
investigated. Exponential fast convergence is obtained "far away" from the
interface. In particular, the exponential decay at regime transitions -
where the number of conditions on the indata needed to obtain well-posedness
changes - for linearized kinetic half-space problems related to the
half-space problem of evaporation and condensation in kinetic theory are
considered. The regime transitions correspond to the transition between
subsonic and supersonic evaporation/condensation, or the transition between
evaporation and condensation. Near the regime transitions, slowly varying
modes might occur, preventing uniform exponential speed of convergence
there. By imposing extra conditions on the indata at the interface, the
slowly varying modes can be eliminated near a regime transition, giving rise
to uniform exponential speed of convergence near the regime transition.}

{\small \ Values of the velocity of the flow at the far end, for which
regime transitions take place are presented for some particular variants of
the Boltzmann equation: for monatomic and polyatomic single species and
mixtures, and the quantum variant for bosons and fermions.}

\section{Introduction with motivating examples \label{S1}}

Half-space problems in the kinetic theory of gases are of great importance
in the study of the asymptotic behavior of solutions of boundary value
problems for the Boltzmann equation for small Knudsen numbers; providing
boundary conditions for the fluid-dynamic-type equations and Knudsen-layer
corrections to solutions of the fluid-dynamic-type equations in a
neighborhood of the boundary \cite{Sone-02, Sone-07}. The steady half-space
problem for the Boltzmann equation in a slab symmetry \cite{Ce-86, Sone-02,
Sone-07} reads%
\begin{equation}
\left\{ 
\begin{array}{l}
v\dfrac{\partial F}{\partial x}=Q\left( F,F\right) ,\;F=F\left( x,\mathbf{v}%
\right) ,\smallskip \\ 
F\left( 0,\mathbf{v}\right) =M_{B}\left( \mathbf{v}\right) \ \text{for}\;v>0,
\\ 
F\rightarrow M_{\infty }\;\text{as}\;x\rightarrow \infty ,%
\end{array}%
\right.  \label{e}
\end{equation}%
where $x\in \mathbb{R}_{+}$ and $\mathbf{v}=\left( v_{1},v_{2},v_{3}\right)
\in \mathbb{R}^{3}$, with $v=v_{1}$. Here $F=F\left( x,\mathbf{v}\right) $
denotes the distribution function of molecules with velocity $\mathbf{v}\in 
\mathbb{R}^{3}$ at distance $x$ from a (planar) interface - typically
between a gas and its condensed phase - and the Gaussian $M_{\infty
}=M_{\infty }\left( \mathbf{v}\right) =\dfrac{\rho }{\left( 2\pi T\right)
^{3/2}}e^{-\left\vert \mathbf{v-u}\right\vert ^{2}/\left( 2T\right) }$
denotes the equilibrium, or, Maxwellian, distribution - approached far away
from the interface; as $x\rightarrow \infty $. Moreover, $\rho \in \mathbb{R}%
_{+}$, $\mathbf{u}=\left( u_{1},u_{2},u_{3}\right) \in \mathbb{R}^{3}$, with 
$u=u_{1}$, and $T\in \mathbb{R}_{+}$, relate (are equal, or, at least
proportional) to the density, bulk velocity, and temperature at the far end,
respectively. The impinging molecules are absorbed at the interface, while
the emerging molecules are desorbed according to a (given) Maxwellian
distribution of the interface (condensed phase). The collision integral $%
Q\left( F,F\right) $, acting on the distribution function $F=F\left( x,%
\mathbf{v}\right) $ only with respect to the velocity dependence, is
quadratic in $F$. The bilinear operator $Q=Q\left( F,G\right) $ is assumed
to be symmetric, such that $Q\left( F,G\right) =Q\left( G,F\right) $ \cite[%
p.11]{Cercignani-00}. After a shift in the velocity space, we obtain, due to
invariance of the collision operator $Q$ under such a transformation,

\begin{equation}
\left\{ 
\begin{array}{l}
\left( v+u\right) \dfrac{\partial \widetilde{F}}{\partial x}=Q\left( 
\widetilde{F},\widetilde{F}\right) ,\smallskip \\ 
\widetilde{F}\left( 0,\mathbf{v}\right) =M_{B}\left( \mathbf{v+u}\right) \ 
\text{for}\;v+u>0, \\ 
\widetilde{F}\rightarrow M\text{ as }x\rightarrow \infty ,%
\end{array}%
\right.  \label{ee}
\end{equation}%
where $\widetilde{F}\left( x,\mathbf{v}\right) =F\left( x,\mathbf{v+u}%
\right) $ and 
\begin{equation*}
M=M\left( \mathbf{v}\right) =M_{\infty }\left( \mathbf{v+u}\right) =\dfrac{%
\rho }{\left( 2\pi RT\right) ^{3/2}}e^{-\left\vert \mathbf{v}\right\vert
^{2}/\left( 2RT\right) }.
\end{equation*}
A suitable linearization, $\widetilde{F}=M+\sqrt{M}f$, around the
non-drifting Maxwellian $M=M_{\infty }\left( \mathbf{v+u}\right) $, results,
after discarding the quadratic terms, in%
\begin{equation}
\left\{ 
\begin{array}{l}
\left( v+u\right) \dfrac{\partial f}{\partial x}+\mathcal{L}f=0,\;f=f\left(
x,\mathbf{v}\right) ,\smallskip \\ 
f\left( 0,\mathbf{v}\right) =M^{-1/2}\left( M_{B}\left( \mathbf{v+u}\right)
-M\right) \ \text{for}\;v+u>0, \\ 
f\rightarrow 0\;\text{as}\;x\rightarrow \infty ,%
\end{array}%
\right.  \label{e0}
\end{equation}%
where $\mathcal{L}f=-2M^{-1/2}Q\left( M,M^{1/2}f\right) $. This problem has
been extensively studied in the literature, see e.g. \cite{BCN-86, Ce-86,
CGS-88, Go-08, GP-89, BGS-06}.

More general boundary conditions, where the distribution function for the
emerging molecules (for which $v>0$ and $v+u>0$, in problem $\left( \ref{e}%
\right) $ and $\left( \ref{ee}\right) $, respectively, for an interface at
rest) may depend (partly, or, completely) on the distribution function for
the impinging molecules (for which $v<0$ and $v+u<0$, in problem $\left( \ref%
{e}\right) $ and $\left( \ref{ee}\right) $, respectively), can\ be
considered at the interface $x=0$ \cite{Cercignani-88, Cercignani-00,
Sone-02, Sone-07}.

Considering the real Hilbert space $\mathfrak{h}_{1}\mathfrak{:}=L^{2}\left(
d\mathbf{v}\right) $, and viewing the molecules as hard spheres, the
linearized operator $\mathcal{L}$ is a nonnegative, self-adjoint Fredholm
operator on $\mathfrak{h}_{1}$, with domain%
\begin{equation*}
\mathrm{D}(\mathcal{L})=L^{2}\left( \left( 1+\left\vert \mathbf{v}%
\right\vert \right) d\mathbf{v}\right)
\end{equation*}%
and kernel 
\begin{equation*}
\ker \mathcal{L}=\mathrm{span}\left\{ \sqrt{M},\sqrt{M}v,\sqrt{M}v_{2},\sqrt{%
M}v_{3},\sqrt{M}\left\vert \mathbf{v}\right\vert ^{2}\right\} .
\end{equation*}

A Fredholm operator $\mathcal{L}$ on a Hilbert space $\mathfrak{h}$ is a
closed operator with finite dimensional kernel and cokernel, and a closed
range. The latter assumption - that the range is closed - is, in fact,
redundant, since any closed operator with a finite dimensional cokernel has
a closed range \cite{Goldberg-66}. Moreover, a self-adjoint operator is, by
definition, densely defined, and, hence, the adjoint operator - and so the
operator itself - is linear and closed \cite{Goldberg-66}. Also, the
orthogonal complement of the range of a closed operator is equal to the
kernel of the adjoint operator \cite{Goldberg-66, Kato}. Hence, for any
self-adjoint operator $\mathcal{L}$ on $\mathfrak{h}$ with a closed range,
the cokernel is equal to the kernel:%
\begin{equation*}
\mathrm{coker}\mathcal{L}=\mathfrak{h/}\mathrm{Im}\mathcal{L}=\left( \mathrm{%
Im}\mathcal{L}\right) ^{\perp }=\ker \mathcal{L}^{\ast }=\ker \mathcal{L}.
\end{equation*}%
The linearized operator $\mathcal{L}$ can be split into a positive
multiplication operator $\nu =\nu (\left\vert \mathbf{v}\right\vert )$ minus
a compact operator $K$ on $\mathfrak{h}_{1}$:%
\begin{equation}
\left( \mathcal{L}f\right) (\mathbf{v})=\nu (\left\vert \mathbf{v}%
\right\vert )f(\mathbf{v})-K(f)(\mathbf{v}),\;f\in \mathrm{D}(\mathcal{L}),
\label{d1}
\end{equation}%
such that for some constants $0<\nu _{-}<\nu _{+}$ 
\begin{equation}
\nu _{-}\left( 1+\left\vert \mathbf{v}\right\vert \right) \leq \nu
(\left\vert \mathbf{v}\right\vert )\leq \nu _{+}\left( 1+\left\vert \mathbf{v%
}\right\vert \right) \text{ for all }\mathbf{v}\in \mathbb{R}^{3}\text{,}
\label{d2}
\end{equation}%
and for some constant $0<\lambda <1$%
\begin{equation}
\int_{\mathbb{R}^{3}}\left( f\mathcal{L}f\right) (\mathbf{v})\ d\mathbf{v}%
\geq \lambda \int_{\mathbb{R}^{3}}\nu (\left\vert \mathbf{v}\right\vert
)f^{2}(\mathbf{v})\ d\mathbf{v}\geq \lambda \nu _{-}\int_{\mathbb{R}%
^{3}}\left( 1+\left\vert \mathbf{v}\right\vert \right) f^{2}(\mathbf{v})\ d%
\mathbf{v}.  \label{d3}
\end{equation}

Turning our attention to the Boltzmann equation for a mixture of $s\geq 2$ ($%
s=1$ corresponds to the case of single species considered above) species $%
\alpha _{1},...,\alpha _{s}$, with masses $m_{\alpha _{1}},...,m_{\alpha
_{s}}$, respectively, the distribution functions (in problem $\left( \ref{e}%
\right) $)\ will be of the form%
\begin{equation*}
F=\left( F_{1},...,F_{s}\right) \text{, where }F_{i}=F_{i}\left( x,\mathbf{v}%
\right) \text{.}
\end{equation*}%
Consider the real Hilbert space $\mathfrak{h}_{s}\mathfrak{:}=\left(
L^{2}\left( d\mathbf{v}\right) \right) ^{s}$, with inner product%
\begin{equation*}
\left( \left. f\right\vert g\right) =\sum_{i=1}^{s}\int_{\mathbb{R}%
^{d}}f_{i}g_{i}\,d\mathbf{v}\text{, }f,g\in \left( L^{2}\left( d\mathbf{v}%
\right) \right) ^{s}\text{,}
\end{equation*}%
and make (after a possible shift in the velocity space, cf. problem $\left( %
\ref{ee}\right) $) a linearization, cf. problem $\left( \ref{e0}\right) $,
around a non-drifting Maxwellian 
\begin{equation*}
M=\left( M_{\alpha _{1}},...,M_{\alpha _{s}}\right) ,\text{ with }M_{\alpha
_{i}}=n_{\alpha _{i}}\left( \dfrac{m_{\alpha _{i}}}{2\pi T}\right)
^{3/2}e^{-m_{\alpha _{i}}\left\vert \mathbf{v}\right\vert ^{2}/\left(
2T\right) },
\end{equation*}%
where $\left\{ n_{\alpha _{1}},...,n_{\alpha _{s}}\right\} \subset \mathbb{R}%
_{+}$ and $T\in \mathbb{R}_{+}$ relate (are equal, or, at least
proportional) to the number densities of the species $\alpha _{1},...,\alpha
_{s}$ and the temperature, respectively. Then the linearized operator $%
\mathcal{L}$ is (considering the molecules to be hard spheres) a
nonnegative, self-adjoint Fredholm operator on $\mathfrak{h}_{s}$ \cite%
{ABT-03, BGPS-13, Be-21a}, with domain%
\begin{equation*}
\mathrm{D}(\mathcal{L})=\left( L^{2}\left( \left( 1+\left\vert \mathbf{v}%
\right\vert \right) d\mathbf{v}\right) \right) ^{s}
\end{equation*}%
and kernel 
\begin{equation*}
\ker \mathcal{L}=\mathrm{span}\left\{ \sqrt{M_{\alpha _{1}}}\mathbf{e}%
_{1},...,\sqrt{M_{\alpha _{s}}}\mathbf{e}_{s},\sqrt{\overline{M}}v,\sqrt{%
\overline{M}}v_{2},\sqrt{\overline{M}}v_{3},\sqrt{\overline{M}}\left\vert 
\mathbf{v}\right\vert ^{2}\right\} ,
\end{equation*}%
where $\overline{M}=\left( m_{\alpha _{1}}^{2}M_{\alpha _{1}},...,m_{\alpha
_{s}}^{2}M_{\alpha _{s}}\right) $ and $\left\{ \mathbf{e}_{1},...,\mathbf{e}%
_{s}\right\} $ is the standard basis of $\mathbb{R}^{s}$. Moreover, the
linearized collision operator$\mathcal{L}=\left( \mathcal{L}_{\alpha
_{1}},...,\mathcal{L}_{\alpha _{s}}\right) $ can be decomposed,
correspondingly to$\ $equations $\left( \ref{d1}\right) $-$\left( \ref{d3}%
\right) $, into a positive multiplication operator $\nu =\nu (\left\vert 
\mathbf{v}\right\vert )=\mathrm{diag}\left( \nu _{1}(\left\vert \mathbf{v}%
\right\vert ),...,\nu _{s}(\left\vert \mathbf{v}\right\vert )\right) $ minus
a compact operator $K=\left( K_{1},...,K_{s}\right) $ on $\mathcal{\mathfrak{%
h}}_{s}$ \cite{BGPS-13}:%
\begin{equation*}
\left( \mathcal{L}f\right) (\mathbf{v})=\nu (\left\vert \mathbf{v}%
\right\vert )f(\mathbf{v})-K(f)(\mathbf{v}),\;f\in \mathrm{D}(\mathcal{L}),
\end{equation*}%
with%
\begin{equation*}
\nu _{-}\left( 1+\left\vert \mathbf{v}\right\vert \right) \leq \nu
_{i}(\left\vert \mathbf{v}\right\vert )\leq \nu _{+}\left( 1+\left\vert 
\mathbf{v}\right\vert \right) \text{ for all }\mathbf{v}\in \mathbb{R}^{3}%
\text{ and }i\in \left\{ 1,...,s\right\} \text{,}
\end{equation*}%
for some constants $0<\nu _{-}<\nu _{+}$ and for some constant $0<\lambda <1$%
\begin{equation*}
\left( \left. f\right\vert \mathcal{L}f\right) \geq \lambda \left( \left.
f\right\vert \nu (\left\vert \mathbf{v}\right\vert )f\right) \geq \lambda
\nu _{-}\left( \left. f\right\vert \left( 1+\left\vert \mathbf{v}\right\vert
\right) f\right) .
\end{equation*}

The remaining of this paper is organized as follows. In Section $\ref{S2}$
we formulate an abstractly formulated half-space problem, motivated by e.g.
the two examples above. The boundary conditions at the interface (at rest)
considered include in addition to those of complete absorption presented
above, much more general ones; cf. the boundary conditions for the
linearized Boltzmann equation presented in \cite[p. 164]{Cercignani-88}. The
main results, including an existence result, Theorem $\ref{T0}$, which tells
that, for a certain number of conditions on the indata at the interface,
there exists a unique solution converging at exponential speed as $%
x\rightarrow \infty $, are presented in Section $\ref{S3}$. In Section $\ref%
{S4}$, a related penalized problem, which is proved to have a unique
solution for any indata, is considered. Sequentially, in Section $\ref{S5}$,
it is proved that under a certain number of conditions on the indata at the
interface, the unique solution to the penalized problem is also a solution
to the original problem. Regime transitions, related to the half-space
problem of evaporation and condensation of gases, are considered in Section $%
\ref{S6}$. The regime transitions, corresponding to the transition between
subsonic and supersonic evaporation/condensation or the transition between
evaporation and condensation, take place at some degenerate values (zero,
or, plus/minus "speed of sound") of the parameter $u$ (the velocity of the
flow - in $x$-direction - at the far end), where the number of conditions,
needed to be imposed on the indata for existence of a unique solution
(stated in Theorem $\ref{T0}$), changes. In general, the exponential decay
is not uniform in any neighborhood of a degenerate value, since, slowly
varying modes may occur as the flow velocity $u$ approaches the degenerate
value (from below). However, by posing some extra condition(s) on the indata
at the interface (before the flow velocity $u$ reaches the degenerate
value), such that the number of conditions on the indata to obtain existence
of a unique solution is the same in a neighborhood of the degenerate value,
the slowly varying modes can be eliminated. Then a uniform exponential decay
in some neighborhood of the degenerate value (Theorems $\ref{T2a}$ and $\ref%
{T2}$) is obtained. In the Appendix the degenerate values of the parameter $%
u $, including the speed of sound, where the regime transitions take place,
together with some important orthogonal basis of the kernel of the
linearized operator (cf. the orthogonality properties $\left( \ref{c2}%
\right) $ below), are presented for some particular variants of the
Boltzmann equation: for monatomic single species and mixtures, as well as,
corresponding cases for polyatomic molecules, and also the quantum variant
for bosons and fermions. The linearized Boltzmann collision operator for
monatomic single species, as well as, mixtures satisfies, for hard spheres,
the assumed properties on the linearized operator in the abstract problem 
\cite{Gr-63, Cercignani-88, BGPS-13,Be-21a}. Some recent corresponding
results for polyatomic molecules can be found in \cite{Be-21a,Be-21b}.

\section{Abstract formulation of the problem \label{S2}}

Let $\mathfrak{h}$ be a real Hilbert space, with inner product $\left(
\left. \cdot \right\vert \cdot \right) $ and denote by $\mathfrak{L}(%
\mathcal{\mathfrak{h}})$ the set of all linear operators on $\mathcal{%
\mathfrak{h}}$. Consider the steady equation 
\begin{equation*}
B\frac{\partial f}{\partial x}+\mathcal{L}f=S,
\end{equation*}%
where $f\equiv f(x,\cdot )\in \mathfrak{h}$ for $x>0$, $S=S(x,\cdot )\in
L^{2}\left( \mathbb{R}_{+};\mathfrak{h}\right) ,\ S=S(x,\cdot )\in \mathrm{Im%
}\mathcal{L}$ belongs to the range $\mathrm{Im}\mathcal{L}$ of $\mathcal{L}$
for all $x\in \mathbb{R}_{+}$. Furthermore, the linear operators $\mathcal{L}%
\ $and $B$, $\left\{ \mathcal{L},B\right\} \subset \mathfrak{L}(\mathcal{%
\mathfrak{h}})$, are assumed to satisfy properties \textbf{H1-H3 }below.

\textbf{H1} The linear operator $\mathcal{L}$ is a nonnegative self-adjoint
Fredholm operator, or, equivalently, $\mathcal{L}$ is a nonnegative
self-adjoint (and, hence, densely defined) operator with a finite
dimensional kernel $\ker \mathcal{L}$ and a closed range $\mathrm{Im}%
\mathcal{L}$.

Then $\mathcal{L}$ has a kernel 
\begin{equation*}
\ker \mathcal{L}=\mathrm{span}\left\{ \phi _{1},...,\phi _{n}\right\} ,\;%
\mathrm{dim}\left( \ker \mathcal{L}\right) =n,
\end{equation*}%
for some $\left\{ \phi _{1},...,\phi _{n}\right\} \subset \mathfrak{h}$; and
a domain $\mathrm{D}(\mathcal{L})$, such that 
\begin{equation*}
\mathcal{L}^{\ast }=\mathcal{L}\geq 0,\text{ }\overline{\mathrm{D}(\mathcal{L%
})}=\mathfrak{h}=\ker \mathcal{L}\oplus \mathrm{Im}\mathcal{L}.
\end{equation*}

\textbf{H2} The linear operator $B$ is a self-adjoint non-singular operator,
such that the domain of $\mathcal{L}$ is a subset of the domain of $B$;%
\begin{equation*}
B^{\ast }=B,\;\ker B=\left\{ 0\right\} ,\;\mathrm{D}(\mathcal{L})\subseteq 
\mathrm{D}(B).
\end{equation*}

Denote by $E\equiv E(d\lambda )$ the spectral measure of $B$ and introduce
the projections 
\begin{equation*}
P_{+}:=\int_{0}^{\infty }E(d\lambda )\text{,}\;P_{-}:=\int_{-\infty
}^{0}E(d\lambda ).
\end{equation*}%
Then the following decomposition may be introduced%
\begin{equation*}
B=B^{+}-B^{-},\;\left\vert B\right\vert =B^{+}+B^{-},\;\text{with }B^{\pm
}=\pm BP_{\pm }.
\end{equation*}%
Furthermore, denote%
\begin{equation*}
\left( \left. \cdot \right\vert \cdot \right) _{\pm }=\left. \left( \left.
\cdot \right\vert \cdot \right) \right\vert _{\mathcal{\mathfrak{h}}_{\pm }},%
\text{ where }\mathcal{\mathfrak{h}}_{\pm }=P_{\pm }\mathcal{\mathfrak{h}}.
\end{equation*}

Remind that, for any closed linear operator $T$ with a closed range, there
exists a positive number $\mu >0$ ("the reduced minimum modulus of $T$" \cite%
[IV-\S 5.1 ]{Kato}) such that \cite{Goldberg-66, Kato}%
\begin{equation*}
\left( \left. Th\right\vert Th\right) \geq \mu \left( \left. h\right\vert
h\right) \;\text{for all}\;h\in \left( \ker T\right) ^{\perp }\cap \mathrm{D}%
(T).
\end{equation*}%
Hence, by assumptions \textbf{H1-H2}, we obtain (by letting $T=\mathcal{L}%
^{1/2}$ - note that $\mathcal{L}^{1/2}$ is a self-adjoint linear operator
with common kernel with $\mathcal{L}$ and a domain containing the domain of $%
\mathcal{L}$; $\ker \mathcal{L}=\ker \mathcal{L}^{1/2}$ and $\mathrm{D}(%
\mathcal{L})\subseteq \mathrm{D}(\mathcal{L}^{1/2})$ \cite{Kato}) that there
exists a positive number $\mu >0$, such that 
\begin{equation}
\left( \left. \mathcal{L}h\right\vert h\right) \geq \mu \left( \left.
h\right\vert h\right) \;\text{for all}\;h\in \mathrm{Im}\mathcal{L}\cap 
\mathrm{D}(\mathcal{L}).  \label{c0}
\end{equation}%
However, here it will be assumed that the operator $\mathcal{L}$ satisfies
an, in general, even stronger condition :

\textbf{H3} There exists a positive number $\gamma >0$, such that $\mathcal{L%
}\geq \gamma \left( 1+\left\vert B\right\vert \right) $ on $\mathrm{Im}%
\mathcal{L}\cap \mathrm{D}(\mathcal{L})$;

\begin{equation}
\left( \left. \mathcal{L}h\right\vert h\right) \geq \gamma \left( \left.
\left( 1+\left\vert B\right\vert \right) h\right\vert h\right) \;\text{for
all}\;h\in \mathrm{Im}\mathcal{L}\cap \mathrm{D}(\mathcal{L}).  \label{c1}
\end{equation}

\begin{remark}
\label{R1a}Assumptions \textbf{H1-H3 }are fulfilled\textbf{\ }for the
linearized Boltzmann collision operator for hard spheres, for (monatomic)
single species, as well as, for (monatomic) multicomponent mixtures, see
Section $\ref{S1}$, with $B=v+u$, where the velocity is given by $\mathbf{v}%
=\left( v,v_{2},v_{3}\right) \in \mathbb{R}^{3}$ and $u\in \mathbb{R}$ is
fixed $\left( \ref{e0}\right) $. Assumptions \textbf{H1-H3 }are also
fulfilled\textbf{\ }for the linearized Boltzmann collision operators for
hard potentials (including hard spheres) if the operator $B$ is bounded,
while this is not the case for soft potentials (the range of the linearized
Boltzmann collision operator for soft potentials is not closed).
\end{remark}

\begin{remark}
\label{R2}If the operator $B$ is bounded, then assumption $\left( \ref{c1}%
\right) $ follows directly by property $\left( \ref{c0}\right) $.
Furthermore, $\mathrm{D}(\mathcal{L})\subseteq \mathfrak{h}=\mathrm{D}(B)$.
\end{remark}

We will consider two different types of boundary conditions at $x=0$.
Introduce operators $P$ and $R$ defined in either of the following two ways
(only considering cases for which such operators exist):

\textbf{H4(a)} $P\in \mathfrak{L}(\mathcal{\mathfrak{h}}_{-}\mathcal{%
\mathfrak{;h}}_{+})$ is a bijective linear operator from $\mathcal{\mathfrak{%
h}}_{-}$ to $\mathcal{\mathfrak{h}}_{+}$, and $R\in \mathfrak{L}(\mathcal{%
\mathfrak{h}}_{+})$ is a linear operator on $\mathcal{\mathfrak{h}}_{+}$,
such that%
\begin{eqnarray}
\left( \left. \left\vert B\right\vert g\right\vert h\right) _{-} &=&\left(
\left. BPg\right\vert Ph\right) _{+},  \notag \\
\left( \left. Rh\right\vert Bg\right) _{+} &=&\left( \left. Bh\right\vert
Rg\right) _{+},  \notag \\
\left( \left. Rg\right\vert BRg\right) _{+} &\leq &\left( \left.
g\right\vert Bg\right) _{+}\text{.}  \label{l1}
\end{eqnarray}

Here and below, we use the simplified notations (for this case):%
\begin{equation*}
Pg=PP_{-}g\text{,}\;Rg=RP_{+}g\text{.}
\end{equation*}

\textbf{H4(b)} $P=\mathbf{1}_{\mathcal{\mathfrak{h}}_{-}}$ is the identity
operator on $\mathcal{\mathfrak{h}}_{-}$, while $R=0\in \mathfrak{L}(%
\mathcal{\mathfrak{h}}_{-}\mathcal{\mathfrak{;h}}_{+})$.

The general formulation of the steady half-space problem of our interest
reads:

\begin{equation}
\left\{ 
\begin{array}{l}
B\dfrac{\partial f}{\partial x}+\mathcal{L}f=S\smallskip \\ 
P_{+}f(0,\cdot )=RPP_{-}f(0,\cdot )+f_{b}%
\end{array}%
\right.  \label{P}
\end{equation}%
for some given $f_{b}\in \mathcal{\mathfrak{h}}_{+}\cap \mathrm{D}(\mathcal{L%
})$, where $e^{\sigma x}f=e^{\sigma x}f(x,\cdot )\in L^{2}\left( \mathbb{R}%
_{+};\mathcal{\mathfrak{h}}\right) $ and $e^{\sigma x}S(x,\cdot )\in
L^{2}\left( \mathbb{R}_{+};\mathcal{\mathfrak{h}}\right) $ for some positive
number $\sigma >0$, and $S=S(x,\cdot )\in \mathrm{Im}\mathcal{L}$ for all $%
x\in \mathbb{R}_{+}$. Substituting%
\begin{equation*}
f=e^{-\sigma x}g,
\end{equation*}%
in problem $\left( \ref{P}\right) $ and introducing the operator%
\begin{equation*}
\widetilde{R}:=P_{+}-RPP_{-}\in \mathfrak{L}\left( \mathfrak{h,h}_{+}\right)
\end{equation*}%
we obtain%
\begin{equation}
\left\{ 
\begin{array}{l}
B\dfrac{\partial g}{\partial x}+\mathcal{L}g-\sigma Bg=e^{\sigma
x}S\smallskip \\ 
\widetilde{R}g(0,\cdot )=f_{b}%
\end{array}%
\right.  \label{LP}
\end{equation}%
for some given $f_{b}\in \mathcal{\mathfrak{h}}_{+}\cap \mathrm{D}(\mathcal{L%
})$, where $g=g(x,\cdot )\in L^{2}\left( \mathbb{R}_{+};\mathcal{\mathfrak{h}%
}\right) $, $S=S(x,\cdot )\in \mathrm{Im}\mathcal{L}$ for all $x\in \mathbb{R%
}_{+}$, and $e^{\sigma x}S(x,\cdot )\in L^{2}\left( \mathbb{R}_{+};\mathcal{%
\mathfrak{h}}\right) $ for some positive number $\sigma >0$.

\begin{remark}
\label{R1}Typically, for Boltzmann(-type) equations (cf. Section $\ref{S1}$
and Remark $\ref{R1a}$), $B=v+u$, while $f=f(x,\mathbf{v})$, with $\mathbf{v}%
=\left( v,v_{2},...,v_{d}\right) \in \mathbb{R}^{d}$ and fixed $u\in \mathbb{%
R}$.

Then $\mathcal{\mathfrak{h_{\pm }}}=\left. \mathcal{\mathfrak{h}}\right\vert
_{v+u\gtrless 0}$ - where, typically, $\mathcal{\mathfrak{h}}=\left(
L^{2}\left( \,d\mathbf{v}\right) \right) ^{s}$ for some positive integer $%
s\geq 1$ -, $Pf(x,\mathbf{v})=f(x,\mathbf{v}_{-})$, with$\;\mathbf{v}_{-}=%
\mathbf{v}-(2\left( v+u\right) ,0,...,0)$, while the linear operator $%
R=R_{u} $ fulfills (here, properties \textbf{H4(a)-(b)} can be combined),
cf. \cite[p. 164]{Cercignani-88} for boundary conditions of linearized
Boltzmann equation (cf. also \cite{LLS-17}) 
\begin{eqnarray}
\left( \left. Rh\right\vert \left( v+u\right) g\right) _{+} &=&\left( \left.
\left( v+u\right) h\right\vert Rg\right) _{+},  \notag \\
\left( \left. Rg\right\vert \left( v+u\right) Rg\right) _{+} &\leq &\left(
\left. g\right\vert \left( v+u\right) g\right) _{+},\;\left( \left. \cdot
\right\vert \cdot \right) _{+}=\left. \left( \left. \cdot \right\vert \cdot
\right) \right\vert _{u+v>0}.  \label{h4}
\end{eqnarray}%
Note that property \textbf{H4(b)} corresponds to complete absorption at the
interface $x=0$.
\end{remark}

\section{Main results \label{S3}}

Let $\left( k^{+},k^{-},l\right) $ be the signature of the restriction of
the quadratic form $\left( \left. B\phi \right\vert \phi \right) $ to the
kernel of $\mathcal{L}$; $k^{+}$, $k^{-}$, and $l$ denote the numbers of
positive, negative, and zero eigenvalues of a symmetric $n\times n$ matrix $%
K $ with elements $k_{ij}=\left( \left. B\phi _{i}\right\vert \phi
_{j}\right) $ for any basis $\left\{ \phi _{1},...,\phi _{n}\right\} $ of
the kernel of $\mathcal{L}$. Due to Sylvester's law of inertia, these
numbers are independent of the choice of basis of the kernel $\ker \mathcal{L%
}$. There exists an orthonormal basis 
\begin{equation*}
\left\{ \phi _{1},...,\phi _{n-l},\psi _{1},...,\psi _{l}\right\}
\end{equation*}%
of the kernel $\ker \mathcal{L}$, such that 
\begin{eqnarray}
\left( \left. \phi _{i}\right\vert \phi _{j}\right) &=&\delta _{ij},\;\left(
\left. \psi _{_{r}}\right\vert \psi _{s}\right) =\delta _{rs},\;\left(
\left. \phi _{i}\right\vert \psi _{r}\right) =0,  \notag \\
\left( \left. B\phi _{i}\right\vert \phi _{j}\right) &=&\beta _{i}\delta
_{ij},\text{ }\left( \left. B\psi _{s}\right\vert \psi _{r}\right) =\left(
\left. B\phi _{i}\right\vert \psi _{r}\right) =0,  \label{c2}
\end{eqnarray}%
with $\beta _{1},...,\beta _{k^{+}}>0$ and $\beta _{k^{+}+1},...,\beta
_{n-l}<0$, cf. \cite{BB-03}. Decompose the kernel of $\mathcal{L}$ in the
following way%
\begin{gather}
\ker \mathcal{L=Z}_{+}\oplus \mathcal{Z}_{-}\oplus \mathcal{Z}_{0}\text{,
where }\mathcal{Z}_{+}:=\mathrm{span}\left\{ \phi _{1},...,\phi
_{k^{+}}\right\} ,\;  \notag \\
\mathcal{Z}_{-}:=\mathrm{span}\left\{ \phi _{k^{+}+1},...,\phi
_{n-l}\right\} ,\;\mathcal{Z}_{0}:=\mathrm{span}\left\{ \psi _{1},...,\psi
_{l}\right\} .  \label{c3}
\end{gather}%
For $\psi \in \mathrm{D}(B)$, there exists $\varphi \in \mathrm{D}(\mathcal{L%
})$, such that $\mathcal{L}\varphi =B\psi $, if and only if 
\begin{equation*}
B\psi \in \mathrm{Im}\mathcal{L}=(\ker \mathcal{L})^{\perp }.
\end{equation*}%
Hence, there exists $\varphi _{r}$ $\in \mathrm{D}(\mathcal{L})$ for each $%
\psi _{r}$, $r\in \left\{ 1,...,l\right\} $, such that%
\begin{equation}
\mathcal{L}\varphi _{r}=B\psi _{r}\text{.}  \label{e10}
\end{equation}%
Without loss of generality, cf. \cite{BB-03}, it can be assumed that%
\begin{equation}
B\varphi _{r}\in \mathcal{Z}_{+}^{\perp }\cap \mathcal{Z}_{-}^{\perp }\text{%
, }\left( \left. B\psi _{r}\right\vert \varphi _{s}\right) =\left( \left. 
\mathcal{L}\varphi _{r}\right\vert \varphi _{s}\right) =\alpha _{r}\delta
_{rs}\text{ with }\alpha _{r}>0.  \label{c4}
\end{equation}

\begin{theorem}
\label{T0}Assume that $S=S(x,\cdot )\in \mathrm{Im}\mathcal{L}$ for all $%
x\in \mathbb{R}_{+}$, $e^{\widetilde{\sigma }x}S(x,\cdot )\in L^{2}\left( 
\mathbb{R}_{+};\mathcal{\mathfrak{h}}\right) $ for some $\widetilde{\sigma }%
>0$, $\widetilde{R}\mathcal{Z}_{\pm }\cup \widetilde{R}\mathcal{Z}%
_{0}\subseteq \mathrm{D}(\mathcal{L})$, and $\dim \left( \mathfrak{h}_{+},%
\mathrm{D}(\mathcal{L})\right) >k^{+}+l$. Then there exists a unique
solution $f$ of the problem $\left( \ref{P}\right) $ such that 
\begin{equation*}
e^{\sigma x}f(x,\cdot )\in L^{2}\left( \mathbb{R}_{+};\mathcal{\mathfrak{h}}%
\right) ,
\end{equation*}%
for some $\sigma >0$, assuming $k^{+}+l$ conditions on $f_{b}\in \mathcal{%
\mathfrak{h}}_{+}\cap \mathrm{D}(\mathcal{L})$.
\end{theorem}

Related problems, in some literature referred to as the Milne and Kramer
problems, see \cite{BCN-86} for the Boltzmann equation, can also be
considered:

\begin{corollary}
\label{C1}Assume that $S=S(x,\cdot )\in \mathrm{Im}\mathcal{L}$ for all $%
x\in \mathbb{R}_{+}$, $e^{\widetilde{\sigma }x}S(x,\cdot )\in L^{2}\left( 
\mathbb{R}_{+};\mathcal{\mathfrak{h}}\right) $ for some $\widetilde{\sigma }%
>0$, $\widetilde{R}\mathcal{Z}_{\pm }\cup \widetilde{R}\mathcal{Z}%
_{0}\subseteq \mathrm{D}(\mathcal{L})$, and $\dim \left( \mathfrak{h}_{+},%
\mathrm{D}(\mathcal{L})\right) >k^{+}+l$. Then there exists a unique
solution $f$ of the problem $\left( \ref{P}\right) $ such that 
\begin{equation*}
e^{\sigma x}\left( f(x,\cdot )-f_{\infty }\right) \in L^{2}\left( \mathbb{R}%
_{+};\mathcal{\mathfrak{h}}\right) ,\;\text{with }f_{\infty }=\,\underset{%
x\rightarrow \infty }{\lim }f(x,\cdot )\in \ker \mathcal{L},
\end{equation*}%
for some $\sigma >0$, if the $k^{-}=n-k^{+}-l$ parameters%
\begin{equation*}
\left\{ \left( \left. f_{\infty }\right\vert \phi _{k^{+}+1}\right)
,...,\left( \left. f_{\infty }\right\vert \phi _{n-l}\right) \right\} ,
\end{equation*}%
are prescribed.
\end{corollary}

\begin{corollary}
\label{C2}Let $S\equiv 0$ and assume that $\widetilde{R}\mathcal{Z}_{\pm
}\cup \widetilde{R}\mathcal{Z}_{0}\subseteq \mathrm{D}(\mathcal{L})$ and $%
\dim \left( \mathfrak{h}_{+},\mathrm{D}(\mathcal{L})\right) >k^{+}+l$. Then
there is a unique solution $f$ of the problem $\left( \ref{P}\right) $ such
that 
\begin{gather*}
e^{\sigma x}\left( f(x,\cdot )-f_{\infty }\right) \in L^{2}\left( \mathbb{R}%
_{+};\mathcal{\mathfrak{h}}\right) ,\text{ }f_{\infty }=\widetilde{\,f}%
_{\infty }+xf_{\infty }^{\prime }\text{,\ } \\
f_{\infty }^{\prime }=\,\underset{x\rightarrow \infty }{\lim }f^{\prime
}(x,\cdot )\text{, }\widetilde{\,f}_{\infty }=\,\underset{x\rightarrow
\infty }{\lim }\left( f(x,\cdot )-xf_{\infty }^{\prime }\right) \in \ker 
\mathcal{L},
\end{gather*}%
for some $\sigma >0$, if the $k^{-}+l=n-k^{+}$ parameters%
\begin{equation*}
\left\{ \left( \left. f_{\infty }\right\vert \phi _{k^{+}+1}\right)
,...,\left( \left. f_{\infty }\right\vert \phi _{n-l}\right) ,\left( \left.
f_{\infty }^{\prime }\right\vert \varphi _{1}\right) ,...,\left( \left.
f_{\infty }^{\prime }\right\vert \varphi _{l}\right) \right\} ,
\end{equation*}%
are prescribed.
\end{corollary}

Letting the linear operator $B$ being the identity operator on $\mathcal{%
\mathfrak{h}}$, implies that $\mathcal{\mathfrak{h}}_{+}=\mathcal{\mathfrak{h%
}}$ and $\mathcal{\mathfrak{h}}_{-}=\{0\}$. Then, replacing $x$ by $t$ and
putting $f_{b}=f_{0}$, we obtain a spatially homogeneous Cauchy problem:

\begin{equation}
\left\{ 
\begin{array}{l}
\dfrac{\partial f}{\partial t}+\mathcal{L}f=S,\;t>0\smallskip \\ 
f(0,\cdot )=f_{0}%
\end{array}%
\right. ,\;f=f(t,\cdot )\text{, }S=S(t,\cdot )\text{, }f_{0}\in \mathrm{D}(%
\mathcal{L})\text{.}  \label{SHP}
\end{equation}%
Here property $\left( \ref{c1}\right) $ follows by property $\left( \ref{c0}%
\right) $, and the following result follows:

\begin{corollary}
Let $\mathcal{L}$ be a nonnegative self-adjoint operator with a closed range
and a finite dimensional kernel 
\begin{equation*}
\ker \mathcal{L}=\mathrm{span}\left\{ \phi _{1},...,\phi _{n}\right\} ,
\end{equation*}%
for some $\phi _{1},...,\phi _{n}\in \mathfrak{h}$. Furthermore, let $%
S=S(t,\cdot )\in \mathrm{Im}\mathcal{L}$ for all $t\in \mathbb{R}_{+}$, $e^{%
\widetilde{\sigma }t}S(t,\cdot )\in L^{2}\left( \mathbb{R}_{+};\mathfrak{h}%
\right) $\ for some $\widetilde{\sigma }>0$, and $f_{0}\in \mathrm{D}(%
\mathcal{L})$. Then the\ linearized spatially homogeneous initial value
problem $\left( \ref{SHP}\right) $ has a unique solution $f$ such that $%
e^{\sigma t}f(t,\cdot )\in L^{2}\left( \mathbb{R}_{+};\mathcal{\mathfrak{h}}%
\right) $ for some $\sigma >0$ if and only if 
\begin{equation*}
\left( \left. f_{0}\right\vert \phi _{1}\right) =...=\left( \left.
f_{0}\right\vert \phi _{n}\right) =0.
\end{equation*}
\end{corollary}

For (variants of) the Boltzmann equation, cf. Remark $\ref{R1}$, a typical
steady half-space problem reads 
\begin{gather}
\left( v+u\right) \frac{\partial f_{u}}{\partial x}+\mathcal{L}%
f_{u}=S_{u},\;f_{u}=f_{u}(x,\mathbf{v}),\;S_{u}=S_{u}(x,\mathbf{v}),  \notag
\\
f_{u}(0,\mathbf{v})=Rf_{u}(0,\mathbf{v}_{-})+f_{bu}(\mathbf{v})\text{\ for\ }%
v+u>0,\;\mathbf{v}_{-}=\mathbf{v}-(2\left( v+u\right) ,0,...,0),  \notag \\
e^{\sigma _{u}x}f_{u}(x,\mathbf{v})\in L^{2}\left( \mathbb{R}_{+};\left(
L^{2}\left( d\mathbf{v}\right) \right) ^{s}\right) \ \text{and}\;  \notag \\
e^{\sigma _{u}x}S_{u}(x,\mathbf{v})\in L^{2}\left( \mathbb{R}_{+};\left(
L^{2}\left( d\mathbf{v}\right) \right) ^{s}\right) \text{ for some }\sigma
_{u}>0,  \notag \\
S_{u}=S_{u}(x,\mathbf{v})\in \mathrm{Im}\mathcal{L}\text{ for all }x\in 
\mathbb{R}_{+}  \label{P2}
\end{gather}%
for some fixed positive integer $s\geq 1$, $\mathbf{v}=(v,v_{2},...,v_{d})%
\in \mathbb{R}^{d}$, and $u\in \mathbb{R}$. Assuming that the linear
operator $R=R_{u}$ fulfills property $\left( \ref{h4}\right) $ in Remark $%
\ref{R1}$ and properties \textbf{H1} and \textbf{H3} being fulfilled, see
Remark $\ref{R1a}$ for some important cases, Theorem $\ref{T0}$ is
applicable, as well as, Corollaries $\ref{C1}$ and $\ref{C2}$.

The theory could be applied to steady half-space problems of the Boltzmann
equation for hard spheres, for monatomic single species \cite{BCN-86,
GPS-88, UYY-03, Go-08}, as well as, for monatomic binary mixtures \cite%
{ABT-03, BY-12} (note that the stated results in \cite{BY-12} are for\ two
species with equal mass), which in a natural way can be extended to
monatomic multicomponent mixtures (cf. \cite{BGPS-13, BD-16, DJMZ-16, Be-21a}%
). It can also be applied for quantum Boltzmann equations as the one for
excitations near a Bose condensate \cite{AN-13}. Recent results for
polyatomic molecules where the polyatomicity is modelled by either a
discrete, or a continuous, internal energy variable, show that for some
particular collision kernels the linearized collision operators fulfill the
assumed properties also in those cases \cite{Be-21a,Be-21b}. Favorable
applications to discrete velocity Boltzmann models for single species and
multicomponent mixtures (of monatomic molecules, as well as polyatomic
molecules), and quantum extensions for Bosons, Fermions, as well as anyons,
cf. \cite{BB-03,Be-08,Be-15,BV-16,Be-16b,Be-17,Be-18,Be-19}, can be stressed
as well.

\begin{remark}
The theory is also applicable for a BGK-like hard-sphere model of the
Boltzmann equation, cf. \cite[p. 96-97, 208]{Cercignani-88} with the
collision frequency $\nu =\nu (\left\vert \mathbf{v}\right\vert )\ $for hard
spheres, for which%
\begin{equation*}
\mathcal{L}f=\nu (\left\vert \mathbf{v}\right\vert )\left(
f-\sum_{i=1}^{d+2}\left( \left. \nu f\right\vert \phi _{i}\right) \phi
_{i}\right) \text{,}
\end{equation*}%
where $\phi _{1},...,\phi _{d+2}$ is an orthonormal, with respect to $\left(
\left. \nu \phi \right\vert \phi \right) $, basis of \linebreak $\left\{ 
\sqrt{M},\sqrt{M}v,\sqrt{M}v_{2},...,\sqrt{M}v_{d},\sqrt{M}\left\vert 
\mathbf{v}\right\vert ^{2}\right\} $;%
\begin{eqnarray*}
\mathrm{span}\left\{ \phi _{1},...,\phi _{d+2}\right\} &=&\mathrm{span}%
\left\{ \sqrt{M},\sqrt{M}v,\sqrt{M}v_{2},...,\sqrt{M}v_{d},\sqrt{M}%
\left\vert \mathbf{v}\right\vert ^{2}\right\} ,\text{ } \\
\left( \left. \nu \phi _{i}\right\vert \phi _{j}\right) &=&\delta _{ij}\text{%
.}
\end{eqnarray*}
\end{remark}

As noted, Theorem $\ref{T0}\ $is applicable for problem $\left( \ref{P2}%
\right) $. However, in general $\sigma _{u}$ will depend on $u$: Theorem $%
\ref{T0}$ will provide us with existence and an exponential speed of
convergence\ for fixed $u$, while the exponential speed of convergence will,
in general, not be uniform in $u$. Nevertheless, on any bounded interval,
whose closure does not contain any degenerate value $u=u_{0}$, i.e. on any
bounded interval, such that $l=0$ for all $u\ $in its closure, $\sigma _{u}$
can be chosen uniformly. This will still remain true if the lower end point
of the interval will be a degenerate value $u=u_{0}$, i.e. with $l>0$ for $%
u=u_{0}$.\ Contrary, if the closure of a bounded interval contains a
degenerate value $u=u_{0}$ (other than the lower end point), there will, in
general, be no uniform exponential speed of convergence on the interval. As $%
u$ tends to the degenerate value $u_{0}$ from below, $u\rightarrow u_{0-}$,
then, in general (without imposing some additional conditions on $f_{bu}$) $%
\sigma _{u}$ will tend to zero. For $u$ sufficiently close to a degenerate
value $u_{0}$ from below, it may occur some slowly varying mode(s) (see e.g. 
\cite{BG-21} for a more explicit presentation in a case of $l=1$;
corresponding from the transition from between condensation and
evaporation). The slowly varying mode(s), can be cancelled by posing some -
more precisely $l$ - additional condition(s) at the interface (for $u$ less
than $u_{0}$). The following result can be obtained:

\begin{theorem}
\label{T2a}Let $u=u_{0}$ be a degenerate value of $u$, i.e. such that $l>0$
for $u=u_{0}$, assume that $\dim \left( \mathfrak{h}_{+},\mathrm{D}(\mathcal{%
L})\right) >k_{0}^{+}+l$ - $k_{0}^{+}$ equals $k^{+}$ for $u=u_{0}$ -, and
assume that for all $u$ in a neighborhood of $u_{0}$: $S_{u}=S_{u}(x,\mathbf{%
v})\in \mathrm{Im}\mathcal{L}$ for all $x\in \mathbb{R}_{+}$, $e^{\widetilde{%
\sigma }x}S_{u}(x,\mathbf{v})\in L^{2}\left( \mathbb{R}_{+};\left(
L^{2}\left( d\mathbf{v}\right) \right) ^{s}\right) $ for some $\widetilde{%
\sigma }>0$ and $\widetilde{R}_{u}\mathcal{Z}_{\pm }\cup \widetilde{R}_{u}%
\mathcal{Z}_{0}\subseteq \mathrm{D}(\mathcal{L}\mathbf{1}_{u+v>0})$, with $%
\widetilde{R}_{u}=\mathbf{1}_{u+v>0}-R_{u}P$, where $Pf(x,\mathbf{v})=f(x,%
\mathbf{v}_{-})$, while $\mathcal{Z}_{\pm }$ and $\mathcal{Z}_{0}$ are
defined in $\left( \ref{c3}\right) $ (with $B=v+u_{0}$).

Then there exists a positive number $\delta (u_{0})>0$, such that by
imposing $k_{0}^{+}+l$ conditions on $f_{bu}\in \left( L^{2}\left( \left(
1+\left\vert v\right\vert \right) \mathbf{1}_{u+v>0}\,d\mathbf{v}\right)
\right) ^{s}\cap \mathrm{D}(\mathcal{L}\mathbf{1}_{u+v>0})$, there exists a
family $\left\{ f_{u}\right\} _{\left\vert u-u_{0}\right\vert \leq \delta
(u_{0})}$ of unique solutions $f_{u}$ of the problem $\left( \ref{P2}\right) 
$ such that 
\begin{equation*}
e^{\sigma x}f_{u}(x,\mathbf{v})\in L^{2}\left( \mathbb{R}_{+};\left(
L^{2}\left( d\mathbf{v}\right) \right) ^{s}\right)
\end{equation*}%
for some positive number $\sigma >0$, independent of $u$, if $\left\vert
u-u_{0}\right\vert \leq \delta (u_{0})$.
\end{theorem}

\section{Penalized problem \label{S4}}

Let $\left\{ \phi _{1},...,\phi _{n-l},\psi _{1},...,\psi _{l}\right\} $ be
an orthonormal basis of the kernel of $\mathcal{L}$, such that relations $%
\left( \ref{c2}\right) $ are satisfied, and let $\left\{ \varphi
_{1},...,\varphi _{l}\right\} \subset \mathrm{D}(\mathcal{L})$, be such that%
\begin{equation*}
\mathcal{L}\varphi _{r}=B\psi _{r},
\end{equation*}%
with relations $\left( \ref{c4}\right) $ satisfied. Without loss of
generality, it may be assumed that 
\begin{equation}
\left( \left. B\varphi _{r}\right\vert \varphi _{s}\right) =0\ \text{for all 
}\left\{ r,s\right\} \subseteq \left\{ 1,...,l\right\} .  \label{c5}
\end{equation}%
Indeed, if the contrary, replace $\left\{ \varphi _{1},...,\varphi
_{l}\right\} $ with $\left\{ \widehat{\varphi }_{1},...,\widehat{\varphi }%
_{l}\right\} $, where 
\begin{eqnarray*}
\widehat{\varphi }_{r} &=&\varphi _{r}-\sum_{s=r+1}^{l}\frac{\left( \left.
B\varphi _{r}\right\vert \varphi _{s}\right) }{\alpha _{s}}\psi _{s}-\frac{%
\left( \left. B\varphi _{r}\right\vert \varphi _{r}\right) }{2\alpha _{r}}%
\psi _{r}\text{ for }r\in \left\{ 1,...,l\right\} \text{, with} \\
\alpha _{i} &=&\left( \left. B\psi _{i}\right\vert \varphi _{i}\right)
=\left( \left. L\varphi _{i}\right\vert \varphi _{i}\right) >0,
\end{eqnarray*}%
and relations $\left( \ref{c4}\right) $,$\left( \ref{c5}\right) $ will be
satisfied. Relations $\left( \ref{c5}\right) $ will be of use in the
forthcoming section.

If $l\neq 0$, denote%
\begin{equation*}
\mathbf{\psi }=\left( \psi _{1},...,\psi _{l}\right) ,\text{\ }\mathbf{%
\varphi }=\left( \varphi _{1},...,\varphi _{l}\right) ,
\end{equation*}%
and introduce the symmetric $l\times l$ matrix $\left\langle \mathbf{\psi }%
\otimes \mathbf{\psi }\right\rangle _{B^{2}}$ with elements 
\begin{equation*}
\left( \left\langle \mathbf{\psi }\otimes \mathbf{\psi }\right\rangle
_{B^{2}}\right) _{rs}=\left( \left. B^{2}\psi _{r}\right\vert \psi
_{s}\right) =\left( \left. B\psi _{r}\right\vert B\psi _{s}\right) =\left(
\left. \mathcal{L}\varphi _{r}\right\vert \mathcal{L}\varphi _{s}\right) .
\end{equation*}%
The matrix $\left\langle \mathbf{\psi }\otimes \mathbf{\psi }\right\rangle
_{B^{2}}$ is symmetric, and hence, there exists an orthogonal $l\times l$
matrix $U$, such that, cf. \cite{BB-03},%
\begin{equation*}
\left\langle \widetilde{\mathbf{\psi }}\otimes \widetilde{\mathbf{\psi }}%
\right\rangle _{B^{2}}=\left\langle U\mathbf{\psi }\otimes U\mathbf{\psi }%
\right\rangle _{B^{2}}=U^{T}\left\langle \mathbf{\psi }\otimes \mathbf{\psi }%
\right\rangle _{B^{2}}U=\mathrm{diag}(\gamma _{1},...,\gamma _{l}),
\end{equation*}%
for some real numbers $\gamma _{1},...,\gamma _{l}$, or, equivalently,%
\begin{gather}
\left( \left. B^{2}\widetilde{\psi }_{r}\right\vert \widetilde{\psi }%
_{s}\right) =\left( \left. B\widetilde{\psi }_{r}\right\vert B\widetilde{%
\psi }_{s}\right) =\left( \left. \mathcal{L}\widetilde{\varphi }%
_{r}\right\vert \mathcal{L}\widetilde{\varphi }_{s}\right) =\gamma
_{r}\delta _{rs},  \notag \\
\text{with }\widetilde{\mathbf{\psi }}:=U\mathbf{\psi }=(\widetilde{\psi }%
_{1},...,\widetilde{\psi }_{l})\ \text{and}\;\widetilde{\mathbf{\varphi }}:=U%
\mathbf{\varphi }=\left( \widetilde{\varphi }_{1},...,\widetilde{\varphi }%
_{l}\right) ,  \label{e11}
\end{gather}%
where, without loss of generality,%
\begin{equation}
\gamma _{1}\geq ...\geq \gamma _{l}>0.  \label{e11b}
\end{equation}%
It may be stressed that, by construction, 
\begin{equation}
\widetilde{\mathbf{\psi }}^{T}\widetilde{\mathbf{\psi }}=\mathbf{\psi }%
^{T}U^{T}U\mathbf{\psi =\psi }^{T}\mathbf{\psi }=I_{l}\text{,}  \label{e12}
\end{equation}%
where $I_{l}$ denote the $l\times l$ identity matrix, or, in other words, $%
\left\{ \widetilde{\psi }_{1},...,\widetilde{\psi }_{l}\right\} $ is an
orthonormal basis of $\mathcal{Z}_{0}$.

Otherwise, if $l=0$, let $\gamma _{1}=0$.

Define the linear operators $\Pi _{+}\in \mathfrak{L}(\mathcal{\mathfrak{h;}Z%
}_{+})$ and $\Pi _{0}\in \mathfrak{L}(\mathcal{\mathfrak{h;}Z}_{0})$\ by%
\begin{equation}
\Pi _{+}:=\sum_{i=1}^{k^{+}}\left( \left. \cdot \right\vert \phi _{i}\right)
\phi _{i}\text{; }\Pi _{0}:=\sum_{r=1}^{l}\frac{\left( \left. \cdot
\right\vert \varphi _{r}\right) }{\left( \left. \varphi _{r}\right\vert 
\mathcal{L}\varphi _{r}\right) ^{2}}\varphi _{r}\text{,}  \label{pr}
\end{equation}%
and consider the penalized problem,\textbf{\ }cf. \cite{UYY-03, Go-08, BG-21}%
:\textbf{\ }%
\begin{equation*}
\left\{ 
\begin{array}{l}
B\dfrac{\partial g}{\partial x}+\mathcal{L}g-\sigma Bg=e^{\sigma x}S-\alpha
\Pi _{+}\left( Bg\right) -\beta B\Pi _{0}\left( Bg\right) ,\;x>0\text{,}%
\smallskip \\ 
\widetilde{R}g(0,\cdot )=f_{b},\text{ }\widetilde{R}=P_{+}-RPP_{-},%
\end{array}%
\right.
\end{equation*}%
or, equivalently, by introducing the linear operator%
\begin{equation}
\Lambda :=\mathcal{L}-\sigma B+\alpha \Pi _{+}B+\beta B\Pi _{0}B,  \label{g1}
\end{equation}%
\begin{equation}
\left\{ 
\begin{array}{l}
B\dfrac{\partial g}{\partial x}+\Lambda g=e^{\sigma x}S,\;x>0\text{,}%
\smallskip \\ 
\widetilde{R}g(0,\cdot )=f_{b},\text{ }\widetilde{R}=P_{+}-RPP_{-}.%
\end{array}%
\right. \text{. }  \label{PP}
\end{equation}

\begin{lemma}
\label{L1}For appropriately chosen positive constants $\alpha $, $\beta $,
and $\sigma $ the operators $\Lambda $ and $\Lambda ^{\star }$ are coercive
on $\mathrm{D}(\mathcal{L})$. Indeed, there exists a positive number $\mu
=\mu \left( \alpha ,\beta ,\sigma \right) >0$ such that%
\begin{equation*}
\left( \left. \Lambda ^{\star }f\right\vert f\right) =\left( \left. \Lambda
f\right\vert f\right) \geq \mu \left( \left. f\right\vert f\right) \ \text{%
for all}\;f\in \mathrm{D}(\mathcal{L}).
\end{equation*}
\end{lemma}

\begin{proof}
Firstly, decompose $f$ orthogonally as%
\begin{equation*}
f=h+q\text{, }q\in \ker \mathcal{L},\;h\in \mathrm{Im}\mathcal{L},
\end{equation*}%
and then, cf. decomposition $\left( \ref{c3}\right) $, $q$ as 
\begin{equation*}
q=q_{+}+q_{-}+q_{0}\text{, }q_{\pm }\in \mathcal{Z}_{\pm },\;q_{0}\in 
\mathcal{Z}_{0}.
\end{equation*}

The operators $\mathcal{L}\ $and $B$, as well as $\Pi _{0}$ and $\Pi _{+}$,
are all self-adjoint, and hence, the adjoint operator of $\Lambda $ $\left( %
\ref{g1}\right) $ equals 
\begin{equation*}
\Lambda ^{\star }=\mathcal{L}-\sigma B+\alpha B\Pi _{+}+\beta B\Pi _{0}B%
\text{.}
\end{equation*}

Let $0<\varepsilon _{1},\varepsilon _{2}<1$. Then%
\begin{eqnarray*}
2\left( \left. h\right\vert Bq_{0}\right) &=&\frac{1}{\varepsilon _{1}^{2}}%
\left( \left. h\right\vert h\right) -\left( \left. \frac{h}{\varepsilon _{1}}%
-\varepsilon _{1}Bq_{0}\right\vert \frac{h}{\varepsilon _{1}}-\varepsilon
_{1}Bq_{0}\right) +\varepsilon _{1}^{2}\left( \left. Bq_{0}\right\vert
Bq_{0}\right) \\
&\leq &\frac{1}{\varepsilon _{1}^{2}}\left( \left. h\right\vert h\right)
+\varepsilon _{1}^{2}\left( \left. Bq_{0}\right\vert Bq_{0}\right)
\end{eqnarray*}%
and%
\begin{eqnarray*}
&&\left( \left. \Pi _{0}\left( B\left( h+q_{0}\right) \right) \right\vert
B\left( h+q_{0}\right) \right) \\
&=&\left( 1-\frac{1}{\varepsilon _{2}^{2}}\right) \left( \left. \Pi
_{0}\left( Bh\right) \right\vert Bh\right) +\left( \left. \Pi _{0}\left(
B\left( \frac{h}{\varepsilon _{2}}+\varepsilon _{2}q_{0}\right) \right)
\right\vert B\left( \frac{h}{\varepsilon _{2}}+\varepsilon _{2}q_{0}\right)
\right) \\
&&+\left( 1-\varepsilon _{2}^{2}\right) \left( \left. \Pi _{0}\left(
Bq_{0}\right) \right\vert Bq_{0}\right) \\
&\geq &\left( 1-\varepsilon _{2}^{2}\right) \left( \left. \Pi _{0}\left(
Bq_{0}\right) \right\vert Bq_{0}\right) -\frac{1-\varepsilon _{2}^{2}}{%
\varepsilon _{2}^{2}}\left( \left. \Pi _{0}\left( Bh\right) \right\vert
Bh\right) \text{.}
\end{eqnarray*}%
Therefore, with $\alpha =2\sigma $ and $0<\varepsilon _{1},\varepsilon
_{2}<1 $,%
\begin{align}
& \left( \left. \Lambda ^{\star }f\right\vert f\right) =\left( \left.
\Lambda f\right\vert f\right)  \notag \\
=& \left( \left. \mathcal{L}h\right\vert h\right) -\sigma \left( \left.
Bh\right\vert h\right) +\left( \alpha -2\sigma \right) \left( \left.
Bh\right\vert q_{+}\right) +\left( \alpha -\sigma \right) \left( \left.
Bq_{+}\right\vert q_{+}\right) -2\sigma \left( \left. Bh\right\vert
q_{-}\right)  \notag \\
& -\sigma \left( \left. Bq_{-}\right\vert q_{-}\right) -2\sigma \left(
\left. h\right\vert Bq_{0}\right) +\beta \left( \left. \Pi _{0}\left(
B\left( h+q_{0}\right) \right) \right\vert B\left( h+q_{0}\right) \right) 
\notag \\
\geq & \left( \left. \mathcal{L}h\right\vert h\right) -\sigma \left( \left.
Bh\right\vert h\right) +\sigma \left( \left. Bq_{+}\right\vert q_{+}\right)
-2\sigma \left( \left. Bh\right\vert q_{-}\right) -\sigma \left( \left.
Bq_{-}\right\vert q_{-}\right) -\frac{\sigma }{\varepsilon _{1}^{2}}\left(
\left. h\right\vert h\right)  \notag \\
& -\sigma \varepsilon _{1}^{2}\left( \left. Bq_{0}\right\vert Bq_{0}\right) -%
\frac{1-\varepsilon _{2}^{2}}{\varepsilon _{2}^{2}}\beta \left( \left. \Pi
_{0}\left( Bh\right) \right\vert Bh\right) +\left( 1-\varepsilon
_{2}^{2}\right) \beta \left( \left. \Pi _{0}\left( Bq_{0}\right) \right\vert
Bq_{0}\right) .  \label{e1}
\end{align}%
Using construction $\left( \ref{e11}\right) $,$\left( \ref{e12}\right) $, $%
q_{0}$ can be decomposed in the following two ways: 
\begin{equation*}
q_{0}=\sum_{r=1}^{l}a_{r}\psi _{r}=\sum_{s=1}^{l}\widetilde{a}_{s}\widetilde{%
\psi }_{s},\text{ }\left( \left. q_{0}\right\vert q_{0}\right)
=\sum_{r=1}^{l}a_{r}^{2}=\sum_{s=1}^{l}\widetilde{a}_{s}^{2},\;a_{r},%
\widetilde{a}_{s}\in \mathbb{R}.
\end{equation*}%
Then 
\begin{equation*}
Bq_{0}=\sum_{s=1}^{l}\widetilde{a}_{s}B\widetilde{\psi }_{s}=%
\sum_{r=1}^{l}a_{r}B\psi _{r}=\sum_{r=1}^{l}a_{r}\mathcal{L}\varphi _{r}%
\text{.}
\end{equation*}%
By relations $\left( \ref{e11}\right) $, $\left( \ref{e11b}\right) $,
follows that%
\begin{equation}
\left( \left. Bq_{0}\right\vert Bq_{0}\right) =\sum_{j=1}^{l}\gamma _{j}%
\widetilde{a}_{j}^{2}\leq \gamma _{1}\sum_{j=1}^{l}\widetilde{a}%
_{j}^{2}=\gamma _{1}\left( \left. q_{0}\right\vert q_{0}\right) ,  \label{e2}
\end{equation}%
while%
\begin{equation}
\left( \left. \Pi _{0}\left( Bq_{0}\right) \right\vert Bq_{0}\right)
=\sum_{r=1}^{l}a_{r}^{2}\frac{\left( \left. B\psi _{r}\right\vert \varphi
_{r}\right) ^{2}}{\left( \left. \varphi _{r}\right\vert \mathcal{L}\varphi
_{r}\right) ^{2}}=\sum_{r=1}^{l}a_{r}^{2}=\left( \left. q_{0}\right\vert
q_{0}\right) .  \label{e2a}
\end{equation}%
Furthermore, applying the Cauchy-Schwarz inequality gives that%
\begin{equation}
\left( \left. \Pi _{0}\left( Bh\right) \right\vert Bh\right) =\sum_{r=1}^{l}%
\frac{\left( \left. h\right\vert B\varphi _{r}\right) ^{2}}{\left( \left.
\varphi _{r}\right\vert \mathcal{L}\varphi _{r}\right) ^{2}}\leq \left(
\left. h\right\vert h\right) \sum_{r=1}^{l}\frac{\left( \left. B\varphi
_{r}\right\vert B\varphi _{r}\right) }{\left( \left. \varphi _{r}\right\vert 
\mathcal{L}\varphi _{r}\right) ^{2}}.  \label{e2b}
\end{equation}%
Moreover, $q_{-}$ can be decomposed as 
\begin{equation*}
q_{-}=\sum_{i=k^{+}+1}^{n-l}b_{i}\phi _{i},\;\text{with }b_{i}\in \mathbb{R},
\end{equation*}%
resulting in%
\begin{eqnarray*}
&&\frac{1}{2}\left( \left. Bq_{-}\right\vert q_{-}\right) +2\left( \left.
Bh\right\vert q_{-}\right) \\
&=&\sum_{i=k^{+}+1}^{n-l}\left( \frac{b_{i}^{2}}{2}\left( \left. B\phi
_{i}\right\vert \phi _{i}\right) +2b_{i}\left( \left. Bh\right\vert \phi
_{i}\right) \right) \\
&=&\sum_{i=k^{+}+1}^{n-l}\left( \frac{b_{i}^{2}}{2}\left( \left. B\phi
_{i}\right\vert \phi _{i}\right) _{+}+2b_{i}\left( \left. Bh\right\vert \phi
_{i}\right) _{+}+\frac{b_{i}^{2}}{2}\left( \left. B\phi _{i}\right\vert \phi
_{i}\right) _{-}+2b_{i}\left( \left. Bh\right\vert \phi _{i}\right)
_{-}\right) \text{.}
\end{eqnarray*}%
Let $0<\varepsilon <1$. Then, for $i\in \left\{ k^{+}+1,...,n-l\right\} $,%
\begin{eqnarray}
&&\frac{b_{i}^{2}}{2}\left( \left. B\phi _{i}\right\vert \phi _{i}\right)
_{\pm }+2b_{i}\left( \left. Bh\right\vert \phi _{i}\right) _{\pm }  \notag \\
&=&\pm \frac{1}{\varepsilon ^{2}}\left( \left. Bh\right\vert h\right) _{\pm
}\mp \left( \left. B\left( \frac{h}{\varepsilon }\mp \varepsilon b_{i}\phi
_{i}\right) \right\vert \frac{h}{\varepsilon }\mp \varepsilon b_{i}\phi
_{i}\right) _{\pm }+\frac{1\pm 2\varepsilon ^{2}}{2}b_{i}^{2}\left( \left.
B\phi _{i}\right\vert \phi _{i}\right) _{\pm }  \notag \\
&\leq &\frac{1}{\varepsilon ^{2}}\left( \left. Bh\right\vert h\right) _{\pm
}+\frac{1\pm 2\varepsilon ^{2}}{2}b_{i}^{2}\left( \left. B\phi
_{i}\right\vert \phi _{i}\right) _{\pm }  \notag \\
&=&\frac{1}{\varepsilon ^{2}}\left( \left. \left\vert B\right\vert
h\right\vert h\right) _{\pm }\pm \frac{1\pm 2\varepsilon ^{2}}{2}%
b_{i}^{2}\left( \left. \left\vert B\right\vert \phi _{i}\right\vert \phi
_{i}\right) _{\pm }\text{.}  \label{e2c}
\end{eqnarray}%
For $n>k^{+}+l$, denote%
\begin{eqnarray}
\widehat{\beta }_{\max } &=&\underset{k^{+}+1\leq i\leq n-l}{\max }\left\{ 
\frac{\beta _{i}^{-}}{\left\vert \beta _{i}\right\vert }\right\} \geq 1,%
\text{ }\beta _{i}^{-}=\left( \left. \left\vert B\right\vert \phi
_{i}\right\vert \phi _{i}\right) _{-},  \notag \\
\beta _{i} &=&\left( \left. B\phi _{i}\right\vert \phi _{i}\right) =\left(
\left. B\phi _{i}\right\vert \phi _{i}\right) _{+}-\beta _{i}^{-}<0.
\label{e3a}
\end{eqnarray}%
For $i\in \left\{ k^{+}+1,...,n-l\right\} $ follows, since $\beta
_{k^{+}+1},...,\beta _{n-l}<0$, that%
\begin{equation*}
\frac{\beta _{i}}{2}+\dfrac{\beta _{i}+2\beta _{i}^{-}}{4\left( \dfrac{\beta
_{i}^{-}}{\left\vert \beta _{i}\right\vert }-\dfrac{1}{2}\right) }=\beta
_{i}^{-}\dfrac{\beta _{i}+\left\vert \beta _{i}\right\vert }{2\beta
_{i}^{-}-\left\vert \beta _{i}\right\vert }=0\text{,}
\end{equation*}%
and furthermore, by inequality $\left( \ref{e2c}\right) $, for $%
0<\varepsilon \leq \dfrac{1}{2\sqrt{\widehat{\beta }_{\max }-\frac{1}{2}}}$%
\begin{eqnarray}
&&\frac{1}{2}\left( \left. Bq_{-}\right\vert q_{-}\right) +2\left( \left.
Bh\right\vert q_{-}\right)  \notag \\
&\leq &\sum_{i=k^{+}+1}^{n-l}\left[ \frac{1+2\varepsilon ^{2}}{2}%
b_{i}^{2}\left( \beta _{i}+\beta _{i}^{-}\right) -\frac{1-2\varepsilon ^{2}}{%
2}b_{i}^{2}\beta _{i}^{-}+\frac{1}{\varepsilon ^{2}}\left( \left. \left\vert
B\right\vert h\right\vert h\right) \right]  \notag \\
&=&\frac{k^{-}}{\varepsilon ^{2}}\left( \left. \left\vert B\right\vert
h\right\vert h\right) +\sum_{i=k^{+}+1}^{n-l}b_{i}^{2}\left( \beta _{i}\frac{%
1+2\varepsilon ^{2}}{2}+2\varepsilon ^{2}\beta _{i}^{-}\right)  \notag \\
&\leq &\frac{k^{-}}{\varepsilon ^{2}}\left( \left. \left\vert B\right\vert
h\right\vert h\right) +\sum_{i=k^{+}+1}^{n-l}b_{i}^{2}\left( \frac{\beta _{i}%
}{2}+\dfrac{\beta _{i}+2\beta _{i}^{-}}{4\left( \dfrac{\beta _{i}^{-}}{%
\left\vert \beta _{i}\right\vert }-\dfrac{1}{2}\right) }\right)  \notag \\
&=&\frac{k^{-}}{\varepsilon ^{2}}\left( \left. \left\vert B\right\vert
h\right\vert h\right) ,\text{ where }k^{-}=n-k^{+}+l>0.  \label{e3}
\end{eqnarray}%
Otherwise, if $k^{-}=n-k^{+}+l=0$, let $\varepsilon =1$.

Denote%
\begin{equation}
\beta _{\min }=\underset{1\leq i\leq n-l}{\min }\left\{ \left\vert \beta
_{i}\right\vert \right\} >0,\text{ }\beta _{i}=\left( \left. B\phi
_{i}\right\vert \phi _{i}\right) \text{,}  \label{e4}
\end{equation}%
let $\beta =\sigma \dfrac{\beta _{\min }+2\gamma _{1}\varepsilon _{1}^{2}}{%
2\left( 1-\varepsilon _{2}^{2}\right) }$ and%
\begin{equation}
\sigma =\dfrac{\gamma }{\max \left( 1+\dfrac{k^{-}}{\varepsilon ^{2}},\dfrac{%
2}{\varepsilon _{1}^{2}}+\dfrac{\beta _{\min }+2\gamma _{1}\varepsilon
_{1}^{2}}{\varepsilon _{2}^{2}}\sum\limits_{r=1}^{l}\dfrac{\left( \left.
B\varphi _{r}\right\vert B\varphi _{r}\right) }{\alpha _{r}^{2}},\dfrac{%
2\gamma \left( 1-\varepsilon _{2}^{2}\right) \underset{1\leq r\leq l}{\max }%
\left( \alpha _{r}\right) }{\beta _{\min }+2\gamma _{1}\varepsilon _{1}^{2}}%
\right) },  \label{e5}
\end{equation}%
where $\gamma $ is given in relation \textbf{H3 }$\left( \ref{c1}\right) $
and $\alpha _{r}=\left( \left. \varphi _{r}\right\vert \mathcal{L}\varphi
_{r}\right) $ for $r\in \left\{ 1,...,l\right\} $. Note that the latter
argument of the maximum in equality $\left( \ref{e5}\right) $ is in fact not
used here, but will be of use first in Lemma $\ref{L3}$ in the forthcoming
section.

The lemma follows, by inequalities $\left( \ref{e1}\right) $, $\left( \ref%
{e2b}\right) $, $\left( \ref{e3}\right) $, and equality $\left( \ref{e2a}%
\right) $:

\begin{eqnarray*}
\left( \left. \Lambda ^{\star }f\right\vert f\right) &=&\left( \left.
\Lambda f\right\vert f\right) \geq \left( \left. \mathcal{L}h\right\vert
h\right) -\sigma \left( 1+\frac{k^{-}}{\varepsilon ^{2}}\right) \left(
\left. \left\vert B\right\vert h\right\vert h\right) +\sigma \left( \left.
Bq_{+}\right\vert q_{+}\right) \\
&&-\frac{\sigma }{2}\left( \left. Bq_{-}\right\vert q_{-}\right) +\left(
\left( 1-\varepsilon _{2}^{2}\right) \beta -\gamma _{1}\sigma \varepsilon
_{1}^{2}\right) \left( \left. q_{0}\right\vert q_{0}\right) \\
&&-\left( \frac{\sigma }{\varepsilon _{1}^{2}}+\frac{1-\varepsilon _{2}^{2}}{%
\varepsilon _{2}^{2}}\beta \sum_{r=1}^{l}\frac{\left( \left. B\varphi
_{r}\right\vert B\varphi _{r}\right) }{\left( \left. \varphi _{r}\right\vert 
\mathcal{L}\varphi _{r}\right) ^{2}}\right) \left( \left. h\right\vert
h\right) \\
&\geq &\frac{\gamma }{2}\left( \left. h\right\vert h\right) +\left( \gamma
-\sigma \left( 1+\frac{k^{-}}{\varepsilon ^{2}}\right) \right) \left( \left.
\left\vert B\right\vert h\right\vert h\right) +\frac{\sigma \beta _{\min }}{2%
}\left( \left. q\right\vert q\right) \\
&&+\left( \frac{\gamma }{2}-\sigma \left( \frac{1}{\varepsilon _{1}^{2}}+%
\frac{\beta _{\min }+2\gamma _{1}\varepsilon _{1}^{2}}{2\varepsilon _{2}^{2}}%
\sum_{r=1}^{l}\frac{\left( \left. B\varphi _{r}\right\vert B\varphi
_{r}\right) }{\alpha _{r}^{2}}\right) \right) \left( \left. h\right\vert
h\right) \\
&\geq &\frac{\gamma }{2}\left( \left. h\right\vert h\right) +\frac{\sigma
\beta _{\min }}{2}\left( \left. q\right\vert q\right) \geq \mu \left( \left.
f\right\vert f\right) \text{,\ with }2\mu =\min \left( \gamma ,\sigma \beta
_{\min }\right) >0.
\end{eqnarray*}
\end{proof}

Define%
\begin{equation*}
\mathcal{T}g:=B\dfrac{\partial g}{\partial x}+\Lambda g,
\end{equation*}%
with domain, we remind that $\widetilde{R}=P_{+}-RPP_{-}$, 
\begin{equation*}
\mathrm{D}(\mathcal{T}):=\left\{ \left\{ g\left( x,\cdot \right) ,\;B\dfrac{%
\partial g}{\partial x}\left( x,\cdot \right) ,\;\mathcal{L}g\left( x,\cdot
\right) \right\} \subset L^{2}\left( \mathbb{R}_{+};\mathcal{\mathfrak{h}}%
\right) ,\;\widetilde{R}g(0,\cdot )=0\right\} .
\end{equation*}%
Then, it follows that%
\begin{equation*}
\mathcal{T}^{\ast }g:=-B\dfrac{\partial g}{\partial x}+\Lambda ^{\ast
}g,\;\Lambda ^{\star }=\mathcal{L}-\sigma B+\alpha B\Pi _{+}+\beta B\Pi
_{0}B,
\end{equation*}%
with domain 
\begin{equation*}
\mathrm{D}(\mathcal{T}^{\ast }):=\left\{ \left\{ g\left( x,\cdot \right) ,\;B%
\dfrac{\partial g}{\partial x}\left( x,\cdot \right) ,\;\mathcal{L}g\left(
x,\cdot \right) \right\} \subset L^{2}\left( \mathbb{R}_{+};\mathcal{%
\mathfrak{h}}\right) ,\;\widetilde{R}^{\ast }g(0,\cdot )=0\right\} ,
\end{equation*}%
where%
\begin{equation*}
\widetilde{R}^{\ast }=RP_{+}-PP_{-}.
\end{equation*}

\begin{lemma}
\label{L2}For appropriately chosen positive constants $\alpha $, $\beta $,
and $\sigma $, there exists a positive number $\mu =\mu \left( \alpha ,\beta
,\sigma \right) >0$ such that%
\begin{eqnarray*}
\left\Vert \mathcal{T}g\right\Vert _{L^{2}\left( \mathbb{R}_{+};\mathcal{%
\mathfrak{h}}\right) } &\geq &\mu \left\Vert g\right\Vert _{L^{2}\left( 
\mathbb{R}_{+};\mathcal{\mathfrak{h}}\right) }\text{ for all }g\in \mathrm{D}%
(\mathcal{T}), \\
\left\Vert \mathcal{T}^{\ast }g\right\Vert _{L^{2}\left( \mathbb{R}_{+};%
\mathcal{\mathfrak{h}}\right) } &\geq &\mu \left\Vert g\right\Vert
_{L^{2}\left( \mathbb{R}_{+};\mathcal{\mathfrak{h}}\right) }\text{ for all }%
g\in \mathrm{D}(\mathcal{T}^{\ast }).
\end{eqnarray*}%
In particular, 
\begin{equation*}
\ker \mathcal{T}=\left\{ 0\right\} \text{\ and }\mathrm{Im}\mathcal{T}%
=L^{2}\left( \mathbb{R}_{+};\mathcal{\mathfrak{h}}\right) .
\end{equation*}
\end{lemma}

\begin{proof}
Let $g\in \mathrm{D}(\mathcal{T})$. Then $Bg\in L^{2}\left( \mathbb{R}_{+};%
\mathcal{\mathfrak{h}}\right) $. Hence, there exists a sequence $\left(
s_{n}\right) _{n=1}^{\infty }$ of positive real numbers such that $%
s_{n}\rightarrow \infty $ and $Bg\left( s_{n},\cdot \right) \rightarrow 0$
in $L^{2}\left( \mathbb{R}_{+};\mathcal{\mathfrak{h}}\right) $ as $%
n\rightarrow \infty $.

By inequality $\left( \ref{l1}\right) $,%
\begin{eqnarray*}
\left( \left. Bg\left( 0,\cdot \right) \right\vert g\left( 0,\cdot \right)
\right) &=&\left( \left. Bg\left( 0,\cdot \right) \right\vert g\left(
0,\cdot \right) \right) _{+}-\left( \left. \left\vert B\right\vert g\left(
0,\cdot \right) \right\vert g\left( 0,\cdot \right) \right) _{-} \\
&=&\left( \left. BRPg\left( 0,\cdot \right) \right\vert RPg\left( 0,\cdot
\right) \right) _{+}-\left( \left. \left\vert B\right\vert Pg\left( 0,\cdot
\right) \right\vert Pg\left( 0,\cdot \right) \right) _{\pm } \\
&\leq &0,
\end{eqnarray*}%
where the last inner product is chosen to match the choice of $P$ and $R$ in
the boundary conditions \textbf{H4(a)-(b)}; indeed, choose $\left( \left.
\cdot \right\vert \cdot \right) _{+}$ for \textbf{H4(a)}, while $\left(
\left. \cdot \right\vert \cdot \right) _{-}$ is to be chosen for \textbf{%
H4(b)}, where $R=0$.

Thus, by Lemma $\ref{L1}$, there exists a positive number $\mu =\mu \left(
\alpha ,\beta ,\sigma \right) >0$ such that 
\begin{eqnarray*}
&&\left\Vert \mathcal{T}g\right\Vert _{L^{2}\left( \mathbb{R}_{+};\mathcal{%
\mathfrak{h}}\right) }\left\Vert g\right\Vert _{L^{2}\left( \mathbb{R}_{+};%
\mathcal{\mathfrak{h}}\right) } \\
&\geq &\int_{0}^{s_{n}}\left( \left. Tg\left( x,\cdot \right) \right\vert
g\left( x,\cdot \right) \right) \,dx \\
&=&\left( \left. Bg\left( s_{n},\cdot \right) \right\vert g\left(
s_{n},\cdot \right) \right) -\left( \left. Bg\left( 0,\cdot \right)
\right\vert g\left( 0,\cdot \right) \right) +\int_{0}^{s_{n}}\left( \left.
\Lambda g\left( x,\cdot \right) \right\vert g\left( x,\cdot \right) \right)
\,dx \\
&\geq &\left( \left. Bg\left( s_{n},\cdot \right) \right\vert g\left(
s_{n},\cdot \right) \right) +\mu \int_{0}^{s_{n}}\left( \left. g\left(
x,\cdot \right) \right\vert g\left( x,\cdot \right) \right) \,dx.
\end{eqnarray*}%
Taking the limit as $n\rightarrow \infty $, results in the inequality%
\begin{equation*}
\left\Vert \mathcal{T}g\right\Vert _{L^{2}\left( \mathbb{R}_{+};\mathcal{%
\mathfrak{h}}\right) }\left\Vert g\right\Vert _{L^{2}\left( \mathbb{R}_{+};%
\mathcal{\mathfrak{h}}\right) }\geq \mu \left\Vert g\right\Vert
_{L^{2}\left( \mathbb{R}_{+};\mathcal{\mathfrak{h}}\right) }^{2}\text{,}
\end{equation*}%
or, equivalently,%
\begin{equation*}
\left\Vert \mathcal{T}g\right\Vert _{L^{2}\left( \mathbb{R}_{+};\mathcal{%
\mathfrak{h}}\right) }\geq \mu \left\Vert g\right\Vert _{L^{2}\left( \mathbb{%
R}_{+};\mathcal{\mathfrak{h}}\right) }.
\end{equation*}%
The analogous statement for $\mathcal{T}^{\ast }$ is proved in a similar way.

The first inequality in the statement of the lemma implies that $\ker 
\mathcal{T}=\left\{ 0\right\} $, and the second one that $\mathrm{Im}%
\mathcal{T}=L^{2}\left( \mathbb{R}_{+};\mathcal{\mathfrak{h}}\right) $.
\end{proof}

\begin{proposition}
\label{P1}Let $e^{\sigma x}S\left( x,\cdot \right) \in L^{2}\left( \mathbb{R}%
_{+};\mathcal{\mathfrak{h}}\right) $ and assume that $f_{b}\in \mathcal{%
\mathfrak{h}}_{+}\cap \mathrm{D}(\mathcal{L})$. Then there exists a unique
solution $g\left( x,\cdot \right) \in L^{2}\left( \mathbb{R}_{+};\mathcal{%
\mathfrak{h}}\right) $ of the penalized problem $\left( \ref{PP}\right) $,
such that%
\begin{equation*}
\mu \left\Vert g\right\Vert _{L^{2}\left( \mathbb{R}_{+};\mathcal{\mathfrak{h%
}}\right) }\leq \left\Vert e^{\sigma x}S\right\Vert _{L^{2}\left( \mathbb{R}%
_{+};\mathcal{\mathfrak{h}}\right) }+\frac{1}{\sqrt{2\sigma }}\left\Vert
\Lambda f_{b}\right\Vert _{\mathcal{\mathfrak{h}}}+\sqrt{\frac{\sigma }{2}}%
\left\Vert Bf_{b}\right\Vert _{\mathcal{\mathfrak{h}}}\text{.}
\end{equation*}%
\bigskip
\end{proposition}

\begin{proof}
Let $h=g(x,\cdot )-f_{b}e^{-\sigma x}$. Then $h\in \mathrm{D}(\mathcal{T})$
if and only if $g\in \widetilde{\mathrm{D}}(\mathcal{T})$, where%
\begin{equation*}
\widetilde{\mathrm{D}}(\mathcal{T})=\left\{ \left\{ g\left( x,\cdot \right)
,\;B\dfrac{\partial g}{\partial x}\left( x,\cdot \right) ,\;\mathcal{L}%
g\left( x,\cdot \right) \right\} \subset L^{2}\left( \mathbb{R}_{+};\mathcal{%
\mathfrak{h}}\right) ,\;\widetilde{R}g(0,\cdot )=f_{b}\right\} ,
\end{equation*}%
and%
\begin{equation*}
\mathcal{T}h=S+\sigma Bf_{b}e^{-\sigma x}-e^{-\sigma x}\Lambda f_{b}\in
L^{2}\left( \mathbb{R}_{+};\mathcal{\mathfrak{h}}\right)
\end{equation*}%
if and only if $g$ is a solution of the penalized problem $\left( \ref{PP}%
\right) $. However, by Lemma $\ref{L2}$, this problem has a unique solution
in $L^{2}\left( \mathbb{R}_{+};\mathcal{\mathfrak{h}}\right) $. Hence, there
exists a unique solution of the penalized problem $\left( \ref{PP}\right) $.

Moreover,%
\begin{eqnarray*}
\mu \left\Vert g\right\Vert _{L^{2}\left( \mathbb{R}_{+};\mathcal{\mathfrak{h%
}}\right) } &\leq &\left\Vert S\right\Vert _{L^{2}\left( \mathbb{R}_{+};%
\mathcal{\mathfrak{h}}\right) }+\left( \left\Vert \Lambda f_{b}\right\Vert _{%
\mathcal{\mathfrak{h}}}+\sigma \left\Vert Bf_{b}\right\Vert _{\mathcal{%
\mathfrak{h}}}\right) \left\Vert e^{-\sigma x}\right\Vert _{L^{2}} \\
&\leq &\left\Vert S\right\Vert _{L^{2}\left( \mathbb{R}_{+};\mathcal{%
\mathfrak{h}}\right) }+\frac{1}{\sqrt{2\sigma }}\left\Vert \Lambda
f_{b}\right\Vert _{\mathcal{\mathfrak{h}}}+\sqrt{\frac{\sigma }{2}}%
\left\Vert Bf_{b}\right\Vert _{\mathcal{\mathfrak{h}}}.
\end{eqnarray*}
\end{proof}

\section{Removal of the penalization\label{S5}}

Denote, with $\varphi _{1},...,\varphi _{l}$ given by equations $\left( \ref%
{e10}\right) ,\left( \ref{c4}\right) ,\left( \ref{c5}\right) $, 
\begin{equation*}
\mathcal{W}_{0}:=\mathrm{span}\left\{ \varphi _{1},...,\varphi _{l}\right\} 
\text{,}
\end{equation*}%
and let $S=S\left( x,\cdot \right) \in \mathrm{Im}\mathcal{L}$ for all $x\in 
\mathbb{R}_{+}$.

An initial step is to transform the problem $\left( \ref{LP}\right) $, in a
way such that $e^{\sigma x}S$ is replaced with $\widetilde{S}=\widetilde{S}%
(x,\cdot )\in \mathrm{Im}\mathcal{L\cap W}_{0}^{\perp }$. Indeed,
substituting%
\begin{equation}
\widetilde{g}(x,\cdot )=g(x,\cdot )+e^{\sigma x}\sum_{r=1}^{l}\psi
_{r}\int_{x}^{\infty }\frac{\left( \left. S\left( \tau ,\cdot \right)
\right\vert \varphi _{r}\right) }{\left( \left. \varphi _{r}\right\vert 
\mathcal{L}\varphi _{r}\right) }\,d\tau  \label{tr1}
\end{equation}%
in problem $\left( \ref{LP}\right) $ results in%
\begin{equation}
\left\{ 
\begin{array}{l}
B\dfrac{\partial \widetilde{g}}{\partial x}+\mathcal{L}\widetilde{g}-\sigma B%
\widetilde{g}=\widetilde{S},\;x>0,\smallskip \\ 
\widetilde{R}\widetilde{g}(0,\cdot )=g_{b},\text{ }\widetilde{R}%
=P_{+}-RPP_{-},%
\end{array}%
\right.  \label{NLP}
\end{equation}%
where, under the assumption that $\widetilde{R}\mathcal{Z}_{0}\subseteq 
\mathrm{D}(\mathcal{L})$,%
\begin{gather*}
\widetilde{S}=\widetilde{S}(x,\cdot )=e^{\sigma x}\left( S(x,\cdot
)-\sum_{r=1}^{l}B\psi _{r}\frac{\left( \left. S(x,\cdot )\right\vert \varphi
_{r}\right) }{\left( \left. \varphi _{r}\right\vert \mathcal{L}\varphi
_{r}\right) }\right) \in \mathrm{Im}\mathcal{L\cap W}_{0}^{\perp }\text{ for
all }x\in \mathbb{R}_{+} \\
\widetilde{S}(x,\cdot )\in L^{2}\left( \mathbb{R}_{+};\mathcal{\mathfrak{h}}%
\right) ,\text{ }g_{b}=f_{b}+\sum_{r=1}^{l}\widetilde{R}\psi
_{r}\int_{0}^{\infty }\frac{\left( \left. S\left( \tau ,\cdot \right)
\right\vert \varphi _{r}\right) }{\left( \left. \varphi _{r}\right\vert 
\mathcal{L}\varphi _{r}\right) }\,d\tau \in \mathcal{\mathfrak{h}}_{+}\cap 
\mathrm{D}(\mathcal{L}).
\end{gather*}%
Therefore, we may, without loss of generality, consider the problem $\left( %
\ref{LP}\right) $, as well as the penalized problem $\left( \ref{PP}\right) $%
, assuming that $S\in \mathrm{Im}\mathcal{L\cap W}_{0}^{\perp }$ for all $%
x\in \mathbb{R}_{+}$.

Denote by $\mathcal{I}:$ $\mathcal{\mathfrak{h}}_{+}\rightarrow \mathcal{%
\mathfrak{h}}$ the solution operator%
\begin{equation}
\mathcal{I}(f_{b})=g\left( 0,\cdot \right) ,  \label{sol1}
\end{equation}%
where $g\left( x,\cdot \right) \in L^{2}\left( \mathbb{R}_{+};\mathcal{%
\mathfrak{h}}\right) $ is the unique solution of the penalized problem $%
\left( \ref{PP}\right) $ in Proposition $\ref{P1}$, and by $\mathbb{I}:$ $%
\mathcal{\mathfrak{h}}_{+}\rightarrow \mathcal{\mathfrak{h}}$ the linear
solution operator%
\begin{equation}
\mathbb{I}(f_{b})=g\left( 0,\cdot \right) ,  \label{lso}
\end{equation}%
in the particular case where $g\left( x,\cdot \right) \in L^{2}\left( 
\mathbb{R}_{+};\mathcal{\mathfrak{h}}\right) $ is the unique solution in
Proposition $\ref{P1}$ of the homogeneous penalized problem%
\begin{equation}
\left\{ 
\begin{array}{l}
B\dfrac{\partial g}{\partial x}+\Lambda g=0,\;x>0\text{,}\smallskip \\ 
\widetilde{R}g(0,\cdot )=f_{b},\text{ }\widetilde{R}=P_{+}-RPP_{-}.%
\end{array}%
\right. \text{.}  \label{HPP}
\end{equation}%
Observe the relation%
\begin{equation*}
\mathcal{I}(f_{b}+\widetilde{f_{b}})=\mathcal{I}(f_{b})+\mathbb{I}(%
\widetilde{f_{b}})\text{ }
\end{equation*}%
for $\left\{ f_{b},\widetilde{f_{b}}\right\} \subset \mathcal{\mathfrak{h}}%
_{+}\cap \mathrm{D}(\mathcal{L})$.

\begin{lemma}
\label{L3}Let $S=S\left( x,\cdot \right) \in \mathrm{Im}\mathcal{L\cap W}%
_{0}^{\perp }$ for all $x\in \mathbb{R}_{+}$ and $e^{\widetilde{\sigma }%
x}S(x,\cdot )\in L^{2}\left( \mathbb{R}_{+};\mathcal{\mathfrak{h}}\right) $
for some $\widetilde{\sigma }>0$, and assume that $f_{b}\in \mathcal{%
\mathfrak{h}}_{+}\cap \mathrm{D}(\mathcal{L})$. Then the solution of the
penalized problem $\left( \ref{PP}\right) $ is a solution of the problem $%
\left( \ref{LP}\right) $ if and only if%
\begin{equation*}
\Pi _{+}\left( B\mathcal{I}(f_{b})\right) =\Pi _{0}\left( B\mathcal{I}%
(f_{b})\right) =0,
\end{equation*}%
or, equivalently, if and only if 
\begin{equation*}
\Pi _{+}\left( B\mathcal{I}(f_{b})\right) =\widetilde{\Pi }_{0}\left( B%
\mathcal{I}(f_{b})\right) =0,
\end{equation*}%
where the projection $\widetilde{\Pi }_{0}\in \mathfrak{L}(\mathcal{%
\mathfrak{h;}Z}_{0})$ on $\mathcal{Z}_{0}$ is the linear operator 
\begin{equation}
\widetilde{\Pi }_{0}=\sum_{s=1}^{l}\left( \left. \cdot \right\vert \psi
_{s}\right) \psi _{s}\text{.}  \label{pr2}
\end{equation}
\end{lemma}

\begin{proof}
Notice that a solution of the penalized problem $\left( \ref{PP}\right) $ is
a solution of problem $\left( \ref{NLP}\right) $ if and only if%
\begin{equation*}
\Pi _{+}\left( Bg\right) =\Pi _{0}\left( Bg\right) =0\text{.}
\end{equation*}%
Assume that $g=g\left( x,\cdot \right) \in L^2 \left( \mathbb{R}_{+};%
\mathcal{\mathfrak{h}}\right) $ is a solution of the penalized problem $%
\left( \ref{PP}\right) $. Then, with $\alpha =2\sigma $ and $\sigma $ given
by expression $\left( \ref{e5}\right) $,%
\begin{equation*}
\dfrac{\partial }{\partial x}\left( \left. Bg\right\vert \phi _{i}\right)
+\sigma \left( \left. Bg\right\vert \phi _{i}\right) =0\text{, }i\in \left\{
1,...,k^{+}\right\} ,
\end{equation*}%
\begin{equation}
\dfrac{\partial }{\partial x}\left( \left. Bg\right\vert \psi _{s}\right)
-\sigma \left( \left. Bg\right\vert \psi _{s}\right) +\beta \frac{\left(
\left. Bg\right\vert \varphi _{s}\right) }{\left( \left. \varphi
_{s}\right\vert \mathcal{L}\varphi _{s}\right) }=0\text{, }s\in \left\{
1,...,l\right\} ,  \label{b0}
\end{equation}%
\begin{equation}
\dfrac{\partial }{\partial x}\left( \left. Bg\right\vert \varphi _{s}\right)
+\left( \left. Bg\right\vert \psi _{s}\right) +2\sigma
\sum_{i=1}^{k^{+}}\left( \left. Bg\right\vert \phi _{i}\right) \left( \left.
\phi _{i}\right\vert \varphi _{s}\right) -\sigma \left( \left. Bg\right\vert
\varphi _{s}\right) =0\text{, }s\in \left\{ 1,...,l\right\} ,  \label{b1}
\end{equation}%
or%
\begin{equation*}
\left( \left. Bg\right\vert \phi _{i}\right) =e^{-\sigma x}\left( \left.
Bg(0,\cdot )\right\vert \phi _{i}\right) \;
\end{equation*}%
for $i\in \left\{ 1,...,k^{+}\right\} $, while for $s\in \left\{
1,...,l\right\} $ 
\begin{eqnarray*}
\left( \left. Bg\right\vert \psi _{s}\right) &=&\beta \int_{x}^{\infty
}e^{\sigma \left( x-\tau \right) }\frac{\left( \left. Bg(\tau ,\cdot
)\right\vert \varphi _{s}\right) }{\left( \left. \varphi _{s}\right\vert 
\mathcal{L}\varphi _{s}\right) }\,d\tau ,\ \text{and} \\
\left( \left. Bg\right\vert \varphi _{s}\right) &=&\int_{x}^{\infty
}e^{\sigma \left( x-\tau \right) }\left( 2\sigma \sum_{i=1}^{k^{+}}\left(
\left. Bg\left( \tau ,\cdot \right) \right\vert \phi _{i}\right) \left(
\left. \phi _{i}\right\vert \varphi _{s}\right) +\left( \left. Bg\left( \tau
,\cdot \right) \right\vert \psi _{s}\right) \right) d\tau \text{.}
\end{eqnarray*}%
However, assuming that $\left( \left. Bg\right\vert \phi _{1}\right)
=...=\left( \left. Bg\right\vert \phi _{k^{+}}\right) =0$, the substitution 
\begin{equation*}
-\beta \frac{\left( \left. Bg\right\vert \varphi _{s}\right) }{\left( \left.
\varphi _{s}\right\vert \mathcal{L}\varphi _{s}\right) }=\dfrac{\partial }{%
\partial x}\left( \left. Bg\right\vert \psi _{s}\right) -\sigma \left(
\left. Bg\right\vert \psi _{s}\right)
\end{equation*}%
in system $\left( \ref{b1}\right) $, results in%
\begin{equation*}
\left( \dfrac{\partial ^{2}}{\partial x^{2}}-2\sigma \dfrac{\partial }{%
\partial x}+\sigma ^{2}-\frac{\beta }{\left( \left. \varphi _{s}\right\vert 
\mathcal{L}\varphi _{s}\right) }\right) \left( \left. Bg\right\vert \psi
_{s}\right) =0\text{,}\;s\in \left\{ 1,...,l\right\} ,
\end{equation*}%
or, correspondingly, the substitution 
\begin{equation*}
-\left( \left. Bg\right\vert \psi _{s}\right) =\dfrac{\partial }{\partial x}%
\left( \left. Bg\right\vert \varphi _{s}\right) -\sigma \left( \left.
Bg\right\vert \varphi _{s}\right)
\end{equation*}%
in system $\left( \ref{b0}\right) $, results in 
\begin{equation*}
\left( \dfrac{\partial ^{2}}{\partial x^{2}}-2\sigma \dfrac{\partial }{%
\partial x}+\sigma ^{2}-\frac{\beta }{\left( \left. \varphi _{s}\right\vert 
\mathcal{L}\varphi _{s}\right) }\right) \left( \left. Bg\right\vert \varphi
_{s}\right) =0\text{,}\;s\in \left\{ 1,...,l\right\} .
\end{equation*}%
Then, with $\beta =\sigma \dfrac{\beta _{\min }+2\gamma _{1}\varepsilon
_{1}^{2}}{2\left( 1-\varepsilon _{2}^{2}\right) }$, it follows that 
\begin{equation*}
\left( \left. Bg\left( x,\cdot \right) \right\vert \psi _{s}\right)
=e^{\left( \sigma -\sqrt{\sigma \tfrac{\beta _{\min }+2\gamma
_{1}\varepsilon _{1}^{2}}{2\left( 1-\varepsilon _{2}^{2}\right) \left(
\left. \varphi _{s}\right\vert \mathcal{L}\varphi _{s}\right) }}\right)
x}\left( \left. Bg\left( 0,\cdot \right) \right\vert \psi _{s}\right) \text{,%
}\;s\in \left\{ 1,...,l\right\} ,
\end{equation*}%
and, correspondingly,%
\begin{equation*}
\left( \left. Bg\left( x,\cdot \right) \right\vert \varphi _{s}\right)
=e^{\left( \sigma -\sqrt{\sigma \tfrac{\beta _{\min }+2\gamma
_{1}\varepsilon _{1}^{2}}{2\left( 1-\varepsilon _{2}^{2}\right) \left(
\left. \varphi _{s}\right\vert \mathcal{L}\varphi _{s}\right) }}\right)
x}\left( \left. Bg\left( 0,\cdot \right) \right\vert \varphi _{s}\right) 
\text{,}\;s\in \left\{ 1,...,l\right\} .
\end{equation*}%
That is, $\Pi _{+}\left( Bg\right) =\Pi _{0}\left( Bg\right) =0$ if and only
if%
\begin{equation*}
\Pi _{+}\left( B\mathcal{I}(f_{b})\right) =\Pi _{0}\left( B\mathcal{I}%
(f_{b})\right) =0
\end{equation*}%
or, equivalently, if and only if%
\begin{equation*}
\Pi _{+}\left( B\mathcal{I}(f_{b})\right) =\widetilde{\Pi }_{0}\left( B%
\mathcal{I}(f_{b})\right) =0.
\end{equation*}
\end{proof}

Note that in the notations of $\left( \ref{tr1}\right) $: $\Pi _{+}\left(
Bg(x,\cdot )\right) =\Pi _{+}\left( B\widetilde{g}(x,\cdot )\right) $
and\linebreak\ $\widetilde{\Pi }_{0}\left( Bg(x,\cdot )\right) =\widetilde{%
\Pi }_{0}\left( B\widetilde{g}(x,\cdot )\right) $, why $\Pi _{+}\left( B%
\mathcal{I}(f_{b})\right) =\Pi _{+}\left( B\widetilde{\mathcal{I}}%
(g_{b})\right) $ and \linebreak $\widetilde{\Pi }_{0}\left( B\mathcal{I}%
(f_{b})\right) =\widetilde{\Pi }_{0}\left( B\widetilde{\mathcal{I}}%
(g_{b})\right) $, where $\mathcal{I}:$ $\mathcal{\mathfrak{h}}%
_{+}\rightarrow \mathcal{\mathfrak{h}}$ and $\widetilde{\mathcal{I}}:$ $%
\mathcal{\mathfrak{h}}_{+}\rightarrow \mathcal{\mathfrak{h}}$ and the
solution operators $\left( \ref{sol1}\right) $ for right hand side $\ S$ and 
$\widetilde{S}$, respectively. Therefore, we can state the following
Corollary.

\begin{corollary}
Let $S=S\left( x,\cdot \right) \in \mathrm{Im}\mathcal{L}$ for all $x\in 
\mathbb{R}_{+}$ and $e^{\widetilde{\sigma }x}S(x,\cdot )\in L^{2}\left( 
\mathbb{R}_{+};\mathcal{\mathfrak{h}}\right) $ for some $\widetilde{\sigma }%
>0$, and assume that $f_{b}\in \mathcal{\mathfrak{h}}_{+}\cap \mathrm{D}(%
\mathcal{L})$. Then the solution of the penalized problem $\left( \ref{PP}%
\right) $ is a solution of the problem $\left( \ref{LP}\right) $ if and only
if 
\begin{equation*}
\Pi _{+}\left( B\mathcal{I}(f_{b})\right) =\widetilde{\Pi }_{0}\left( B%
\mathcal{I}(f_{b})\right) =0,
\end{equation*}%
where $\widetilde{\Pi }_{0}\in \mathfrak{L}(\mathcal{\mathfrak{h;}Z}_{0})$
is the projection $\left( \ref{pr2}\right) $.
\end{corollary}

\begin{theorem}
\label{T1}Let $S=S\left( x,\cdot \right) \in \mathrm{Im}\mathcal{L}$ for all 
$x\in \mathbb{R}_{+}$ and $e^{\widetilde{\sigma }x}S(x,\cdot )\in
L^{2}\left( \mathbb{R}_{+};\mathcal{\mathfrak{h}}\right) $ for some $%
\widetilde{\sigma }>0$, $\widetilde{R}\mathcal{Z}_{\pm }\cup \widetilde{R}%
\mathcal{Z}_{0}\subseteq \mathrm{D}(\mathcal{L})$, and $\dim \left( 
\mathfrak{h}_{+},\mathrm{D}(\mathcal{L})\right) >k^{+}+l$. Then there exists
a unique solution $g\left( x,\cdot \right) \in L^{2}\left( \mathbb{R}_{+};%
\mathcal{\mathfrak{h}}\right) $ of the problem $\left( \ref{LP}\right) $,
such that 
\begin{equation*}
\mu \left\Vert g\right\Vert _{L^{2}\left( \mathbb{R}_{+};\mathcal{\mathfrak{h%
}}\right) }\leq \left\Vert S\right\Vert _{L^{2}\left( \mathbb{R}_{+};%
\mathcal{\mathfrak{h}}\right) }+\frac{1}{\sqrt{2\sigma }}\left\Vert \mathcal{%
L}g_{b}\right\Vert _{\mathcal{\mathfrak{h}}}+\sqrt{2\sigma }\left\Vert
Bg_{b}\right\Vert _{\mathcal{\mathfrak{h}}},
\end{equation*}%
for some $\sigma >0$, assuming 
\begin{equation*}
\mathrm{codim}\left( \left\{ \left. f_{b}\in \mathcal{\mathfrak{h}}_{+}\cap 
\mathrm{D}(\mathcal{L})\,\right\vert \,\Pi _{+}\left( B\mathbb{I}%
(f_{b})\right) =\widetilde{\Pi }_{0}\left( B\mathbb{I}(f_{b})\right)
=0\right\} \right) =k^{+}+l
\end{equation*}%
conditions on $f_{b}\in \mathcal{\mathfrak{h}}_{+}\cap \mathrm{D}(\mathcal{L}%
)$.
\end{theorem}

\begin{proof}
Denote 
\begin{equation*}
\mathcal{P}=\left\{ \left. f_{b}\in \mathcal{\mathfrak{h}}_{+}\cap \mathrm{D}%
(\mathcal{L})\,\right\vert \,\Pi _{+}\left( B\mathbb{I}(f_{b})\right) =%
\widetilde{\Pi }_{0}\left( B\mathbb{I}(f_{b})\right) =0\right\} .
\end{equation*}%
Since $\Pi _{+}\in \mathfrak{L}(\mathcal{\mathfrak{h;}Z}_{+})$ and $%
\widetilde{\Pi }_{0}\in \mathfrak{L}(\mathcal{\mathfrak{h;}Z}_{0})$,%
\begin{equation*}
\mathrm{codim}\left( \mathcal{P}\right) \leq \dim \mathcal{Z}_{+}+\dim 
\mathcal{Z}_{0}=k^{+}+l.
\end{equation*}%
Since, by assumption, $\dim \left( \mathcal{\mathfrak{h}}_{+}\cap \mathrm{D}(%
\mathcal{L})\right) >k^{+}+l$ 
\begin{equation*}
\widetilde{\mathcal{P}}_{\widetilde{S}}=\left\{ \left. f_{b}\in \mathcal{%
\mathfrak{h}}_{+}\cap \mathrm{D}(\mathcal{L})\,\right\vert \,\Pi _{+}\left( B%
\mathcal{I}(f_{b})\right) =\widetilde{\Pi }_{0}\left( B\mathcal{I}%
(f_{b})\right) =0\right\}
\end{equation*}%
is non-empty for any right hand side $\widetilde{S}$, such that $e^{%
\widetilde{\sigma }x}S(x,\cdot )\in L^{2}\left( \mathbb{R}_{+};\mathcal{%
\mathfrak{h}}\right) $ for some $\widetilde{\sigma }>0$. Let $\alpha
=2\sigma $ and let $g_{0}=g_{0}\left( x,\cdot \right) $ be the unique
solution of the penalized problem $\left( \ref{PP}\right) $, and therefore
of the undamped problem $\left( \ref{LP}\right) $, as well, for some $f_{b}=$
$g_{b0}\in \widetilde{\mathcal{P}}_{S}$.

Note that $f_{b1}-f_{b2}\in \mathcal{P}$ for $\left\{ f_{b1},f_{b2}\right\}
\subset \widetilde{\mathcal{P}}_{S}$, while $f_{b1}+f_{b2}\in \widetilde{%
\mathcal{P}}_{S}$ if $f_{b1}\in \widetilde{\mathcal{P}}_{S}$ and $f_{b2}\in 
\mathcal{P}$, and consequently,%
\begin{equation*}
\widetilde{\mathcal{P}}_{S}=g_{b0}+\mathcal{P}\text{.}
\end{equation*}

For $i\in \left\{ 1,...,k^{+}\right\} $: let $S_{i}\left( x,\cdot \right)
=2\sigma e^{-2\sigma x}\left( B\phi _{i}-\beta _{i}\phi _{i}\right) \in
L^{2}\left( \mathbb{R}_{+};\mathcal{\mathfrak{h}}\right) $, with $\beta
_{i}=\left( \left. B\phi _{i}\right\vert \phi _{i}\right) >0$, and $%
g_{bi}\in \widetilde{\mathcal{P}}_{S_{i}}$. Moreover, let $g_{i}^{\prime
}\left( x,\cdot \right) \in L^{2}\left( \mathbb{R}_{+};\mathcal{\mathfrak{h}}%
\right) $ be the unique solution of the problem $\left( \ref{PP}\right) $
with $S=S_{i}$ and $f_{b}=g_{bi}$; 
\begin{equation*}
\left\{ 
\begin{array}{l}
B\dfrac{\partial g_{i}^{\prime }}{\partial x}+\Lambda g_{i}^{\prime
}=2\sigma e^{-\sigma x}\left( B\phi _{i}-\beta _{i}\phi _{i}\right)
,\;x>0,\smallskip \\ 
\widetilde{R}g_{i}^{\prime }(0,\cdot )=g_{bi}.%
\end{array}%
\right.
\end{equation*}%
By simple calculations, it can be verified that%
\begin{equation*}
h_{i}(x,\cdot ):=g_{i}^{\prime }(x,\cdot )+e^{-\sigma x}\phi _{i}\in
L^{2}\left( \mathbb{R}_{+};\mathcal{\mathfrak{h}}\right) ,
\end{equation*}%
is the unique solution of the homogeneous penalized problem $\left( \ref{HPP}%
\right) $ with $f_{b}=g_{bi}+\widetilde{R}\phi _{i}$, while 
\begin{equation*}
g_{i}(x,\cdot ):=g_{0}(x,\cdot )+h_{i}(x,\cdot )=g_{0}(x,\cdot
)+g_{i}^{\prime }(x,\cdot )+e^{-\sigma x}\phi _{i}\in L^{2}\left( \mathbb{R}%
_{+};\mathcal{\mathfrak{h}}\right) ,
\end{equation*}%
is the unique solution of the penalized problem $\left( \ref{PP}\right) $
with $f_{b}=g_{b0}+g_{bi}+\widetilde{R}\phi _{i}$. It follows that 
\begin{eqnarray*}
\Pi _{+}\left( Bg_{i}(0,\cdot )\right) &=&\Pi _{+}\left( Bh_{i}(0,\cdot
)\right) =\beta _{i}\phi _{i},\text{\ }\beta _{i}=\left( \left. B\phi
_{i}\right\vert \phi _{i}\right) >0,\; \\
\widetilde{\Pi }_{0}\left( Bg_{i}(0,\cdot )\right) &=&\widetilde{\Pi }%
_{0}\left( Bh_{i}(0,\cdot )\right) =0.
\end{eqnarray*}

For $r\in \left\{ 1,...,l\right\} $: let $\widetilde{S}_{r}\left( x,\cdot
\right) =e^{-2\sigma x}\left( 4\sigma ^{2}\alpha _{r}-1\right) B\psi _{r}\in
L^{2}\left( \mathbb{R}_{+};\mathcal{\mathfrak{h}}\right) $, $\alpha
_{r}=\left( \left. B\psi _{r}\right\vert \varphi _{r}\right) =\left( \left.
L\varphi _{r}\right\vert \varphi _{r}\right) >0$ and $\widetilde{g}_{br}\in 
\mathcal{P}_{\widetilde{S}_{r}}$. Moreover, let $\widetilde{g}_{r}^{\prime
}\left( x,\cdot \right) \in L^{2}\left( \mathbb{R}_{+};\mathcal{\mathfrak{h}}%
\right) $ be the unique solution of the penalized problem $\left( \ref{PP}%
\right) $ with $S=\widetilde{S}_{r}$ and $f_{b}=\widetilde{g}_{br}$; 
\begin{equation*}
\left\{ 
\begin{array}{l}
B\dfrac{\partial \widetilde{g}_{r}^{\prime }}{\partial x}+\Lambda \widetilde{%
g}_{r}^{\prime }=e^{-\sigma x}\left( \dfrac{4\sigma ^{2}\alpha _{r}}{\beta }%
-1\right) B\psi _{r},\;x>0,\smallskip \\ 
\widetilde{R}\widetilde{g}_{r}^{\prime }(0,\cdot )=\widetilde{g}_{br}.%
\end{array}%
\right.
\end{equation*}%
Again by simple calculations, it can be verified that%
\begin{equation*}
\widetilde{h}_{r}(x,\cdot ):=\widetilde{g}_{r}^{\prime }(x,\cdot
)+e^{-\sigma x}\left( \varphi _{r}+\frac{2\sigma \alpha _{r}}{\beta }\psi
_{r}\right) \in L^{2}\left( \mathbb{R}_{+};\mathcal{\mathfrak{h}}\right) ,
\end{equation*}%
is the unique solution of the homogeneous penalized problem $\left( \ref{HPP}%
\right) $ with $f_{b}=\widetilde{g}_{br}+\widetilde{R}\left( \varphi _{r}+%
\dfrac{2\sigma \alpha _{r}}{\beta }\psi _{r}\right) $, while%
\begin{eqnarray*}
\widetilde{g}_{r}(x,\cdot ):= &&g_{0}(x,\cdot )+\widetilde{h}_{r}(x,\cdot )
\\
= &&g_{0}(x,\cdot )+\widetilde{g}_{r}^{\prime }(x,\cdot )+e^{-\sigma
x}\left( \varphi _{r}+\dfrac{2\sigma \alpha _{r}}{\beta }\psi _{r}\right)
\in L^{2}\left( \mathbb{R}_{+};\mathcal{\mathfrak{h}}\right) ,
\end{eqnarray*}%
is the unique solution of the penalized problem $\left( \ref{PP}\right) $
with $f_{b}=g_{b0}+\widetilde{g}_{br}+\widetilde{R}\left( \varphi _{r}+%
\dfrac{2\sigma \alpha _{r}}{\beta }\psi _{r}\right) $. It follows that%
\begin{eqnarray*}
\Pi _{+}\left( B\widetilde{g}_{r}(0,\cdot )\right) &=&\Pi _{+}\left( B%
\widetilde{h}_{r}(0,\cdot )\right) =0, \\
\widetilde{\Pi }_{0}\left( B\widetilde{g}_{r}(0,\cdot )\right) &=&\widetilde{%
\Pi }_{0}\left( B\widetilde{h}_{r}(0,\cdot )\right) =\alpha _{r}\psi _{r}.
\end{eqnarray*}

Consequently,%
\begin{equation*}
\mathrm{codim}\left( \mathcal{P}\right) =k^{+}+l,
\end{equation*}%
following by the uniqueness of solutions to the homogeneous penalized
problem $\left( \ref{HPP}\right) $ and the linear independence of $\left\{
\phi _{1},...,\phi _{k^{+}},\psi _{1},...,\psi _{l}\right\} $.
\end{proof}

\section{Half-space problem of evaporation and condensation - regime
transitions \label{S6}}

This section concerns Boltzmann(-type) equations $\left( \ref{P2}\right) $,
assuming properties \textbf{H1 }and\textbf{\ H3,} as well as \textbf{H4} in
the form $\left( \ref{h4}\right) $, cf. Remark $\ref{R1}$. Therefore,
Theorem $\ref{T0}\ $is applicable.

In general, the exponential speed of convergence - $\sigma _{u}>0$ in
problem $\left( \ref{P2}\right) $ - depends on $u$. Theorem $\ref{T0}$
provides existence of a unique solution and an exponential speed of
convergence for fixed $u$. However, on an interval of $u$ the exponential
speed of convergence may not be uniform - there might occur slowly varying
modes in some regions of $u$ (cf. \cite{BG-21} and references therein).
Nevertheless, on any bounded interval, whose closure does not contain any
degenerate value $u=u_{0}$, i.e. on any interval such that $l=0$ for all $u$
in the closure, there is a uniform exponential speed of convergence - $%
\sigma _{u}>0$ can be uniformly determined. Remind that $\left(
k^{+},k^{-},l\right) $ denotes the signature of the restriction of the
quadratic form $\left( \left. (v+u)\phi \right\vert \phi \right) $ to the
kernel of $\mathcal{L}$. Moreover, based on the arguments in the previous
sections one can prove that, the slowly varying modes can be eliminated, if
they occur, by imposing extra conditions on the indata at the interface. The
following result can be obtained:

\begin{theorem}
\label{T2}Let $u=u_{0}$ be a degenerate value of $u$, i.e. such that $l>0$
for $u=u_{0}$, assume that $\dim \left( \widehat{\mathcal{\mathfrak{h}}}_{+},%
\mathrm{D}(\mathcal{L})\right) >k_{0}^{+}+l$ - $\widehat{\mathcal{\mathfrak{h%
}}}_{+}=\left( L^{2}\left( \left( 1+\left\vert v\right\vert \right) \mathbf{1%
}_{v+u>0}\,d\mathbf{v}\right) \right) ^{s}\cap \mathrm{D}(\mathcal{L}\mathbf{%
1}_{v+u>0})$, while $k_{0}^{+}$ equals $k^{+}$ for $u=u_{0}$ - and assume
that for all $u$ in a neighborhood of $u_{0}$: $S_{u}=S_{u}(x,\mathbf{v})\in 
\mathrm{Im}\mathcal{L}$ for all $x\in \mathbb{R}_{+}$, $e^{\widetilde{\sigma 
}x}S_{u}(x,\mathbf{v})\in L^{2}\left( \mathbb{R}_{+};\left( L^{2}\left( d%
\mathbf{v}\right) \right) ^{s}\right) $ for some $\widetilde{\sigma }>0$,
and $\widetilde{R}_{u}\mathcal{Z}_{\pm }\cup \widetilde{R}_{u}\mathcal{Z}%
_{0}\subseteq \mathrm{D}(\mathcal{L}\mathbf{1}_{v+u>0})$, with $\widetilde{R}%
_{u}=\mathbf{1}_{v+u>0}-R_{u}P$, where $Pf(x,\mathbf{v})=f(x,\mathbf{v}_{-})$%
, while $\mathcal{Z}_{\pm }$ and $\mathcal{Z}_{0}$ are defined in $\left( %
\ref{c3}\right) $ - for $B=v+u_{0}$. Moreover, let $\Pi _{+}^{0}$ and $%
\widetilde{\Pi }_{0}^{0}$ denote the linear operators $\Pi _{+}$ $\left( \ref%
{pr}\right) $ and $\widetilde{\Pi }_{0}$ $\left( \ref{pr2}\right) $ for $%
u=u_{0}$, $k^{+}=k_{0}^{+}$ for $u=u_{0}$, and $\mathbb{I}$ be the linear
solution operator $\left( \ref{lso}\right) $. Then there exists a positive
number $\delta (u_{0})>0$, such that by posing 
\begin{equation*}
\mathrm{codim}\left( \left\{ \left. f_{bu}\in \widehat{\mathcal{\mathfrak{h}}%
}_{+}\right\vert \;\Pi _{+}^{0}\left( \left( v+u\right) \mathbb{I}%
(f_{bu})\right) =\widetilde{\Pi }_{0}^{0}\left( \left( v+u\right) \mathbb{I}%
\left( f_{bu}\right) \right) =0\right\} \right) =k_{0}^{+}+l,
\end{equation*}%
conditions on $f_{bu}\in \widehat{\mathcal{\mathfrak{h}}}_{+}$, there exists
a family $\left\{ f_{u}\right\} _{\left\vert u-u_{0}\right\vert \leq \delta
(u_{0})}$ of unique solutions $f_{u}=f_{u}(x,\mathbf{v})$ of the problem $%
\left( \ref{P2}\right) $, such that%
\begin{equation*}
e^{\sigma x}f_{u}(x,\mathbf{v})\in L^{2}\left( \mathbb{R}_{+};\left(
L^{2}\left( d\mathbf{v}\right) \right) ^{s}\right)
\end{equation*}%
for some positive number $\sigma >0$, independent of $u$, if $\left\vert
u-u_{0}\right\vert \leq \delta (u_{0})$.
\end{theorem}

\begin{remark}
\label{RH}Let $u=u_{0}$ be a degenerate value of $u$ of order $l$, and let $%
k^{+}=k_{0}^{+}$ for $u=u_{0}$, with 
\begin{eqnarray*}
\phi _{i} &=&\phi _{i0}\ \text{for\ all\ }i\in \left\{
1,...,k_{0}^{+}=k^{+}\right\} , \\
\psi _{s} &=&\psi _{s0}\text{ }\ \text{for\ all}\;s\in \left\{
1,...,l\right\} .
\end{eqnarray*}%
Then, by the orthogonality relations $\left( \ref{c2}\right) $,%
\begin{eqnarray*}
\left( \left. \left( v+u\right) \phi _{i0}\right\vert \phi _{j0}\right)
&=&\left( \beta _{i}^{0}+u-u_{0}\right) \delta _{ij},\;\beta _{i}^{0}=\left(
\left. \left( v+u_{0}\right) \phi _{i0}\right\vert \phi _{i0}\right) >0\text{%
,} \\
\left( \left. \left( v+u\right) \psi _{s0}\right\vert \psi _{r0}\right)
&=&\left( u-u_{0}\right) \delta _{rs},\;\left( \left. \left( v+u\right) \phi
_{i0}\right\vert \psi _{r0}\right) =0.
\end{eqnarray*}%
Hence, for all $u>u_{0}$, such that there is no degenerate value on the
interval $\left( u_{0},u\right] $, $k^{+}=k_{0}^{+}+l$, and it is possible
to choose%
\begin{eqnarray*}
\phi _{i} &=&\phi _{i0}\text{ }\ \text{for\ all }\;i\in \left\{
1,...,k_{0}^{+}=k^{+}-l\right\} , \\
\phi _{k^{+}-l+s} &=&\psi _{s0}\text{ }\ \text{for\ all }\;s\in \left\{
1,...,l\right\} .
\end{eqnarray*}%
That is, we impose no extra conditions - compared to the number of
conditions imposed in Theorem $\ref{T0}$ - on the boundary data as $u>u_{0}$%
. \ On the other hand, for all $u<u_{0}$, such that there is no degenerate
value on the interval $\left[ u,u_{0}\right) $, $k^{+}=k_{0}^{+}$, and it is
possible to choose%
\begin{eqnarray*}
\phi _{i} &=&\phi _{i0}\;\ \text{for\ all }i\in \left\{
1,...,k_{0}^{+}\right\} , \\
\phi _{k^{+}+s} &=&\psi _{s0}\;\ \text{for\ all }s\in \left\{
1,...,l\right\} .
\end{eqnarray*}%
That is, we impose $l$ extra conditions - compared to the number of
conditions imposed in Theorem $\ref{T0}$ - on the boundary data as $u<u_{0}$.
\end{remark}

Note that in the notations of Lemma $\ref{L1}$, any solution of the undamped
problem $\left( \ref{LP}\right) $ - cf. the proof of Lemma $\ref{L3}$ - must
satisfy $q_{+}=q_{0}=0$. Repeating the arguments of the proof of Lemma $\ref%
{L1}$, assuming that $q_{+}=q_{0}=0$, will remove any smallness assumptions
on $\sigma $ or $\mu $ - obtained through large $\widehat{b}_{\max }$ $%
\left( \ref{e3a}\right) $; $\widehat{b}_{\max }\gg 1$, or small $b_{\min }$ $%
\left( \ref{e4}\right) $; $b_{\min }\ll 1$ -, for $u\geq u_{0}$ as long as $%
l=0$ for any $u\in (u_{0},u]$. However, this is not the case as $%
u\rightarrow u_{0}^{-}$. For $u\neq u_{0}$: $\widetilde{\beta }%
_{i}(u):=\left( \left. \left( v+u\right) \phi _{i0}\right\vert \phi
_{i0}\right) =\beta _{i}^{0}+u-u_{0}$ for $i\in \left\{
1,...,k_{0}^{+}\right\} $, while $\widetilde{\beta }_{i}(u):=\left( \left.
\left( v+u\right) \psi _{s0}\right\vert \psi _{s0}\right) =u-u_{0}$ for $%
i\in \left\{ k_{0}^{+}+1,...,k_{0}^{+}+l\right\} $; implying that $%
\widetilde{\beta }_{k_{0}^{+}+1},...,\widetilde{\beta }_{k_{0}^{+}+l}$ tend
to zero as $u$ tends to $u_{0}$, while this will not be the case for $%
\widetilde{\beta }_{1},...,\widetilde{\beta }_{k_{0}^{+}}$. That is, $%
\widetilde{\beta }_{i}\rightarrow 0^{-}$ as $u\rightarrow u_{0}^{-}$ for all 
$i\in \left\{ k_{0}^{+}+1,...,k_{0}^{+}+l\right\} $, while $\widetilde{\beta 
}_{i}\nrightarrow 0$ as $u\rightarrow u_{0}^{-}$ for all $i\in \left\{
1,...,k_{0}^{+}\right\} $. The fact that $\widetilde{\beta }_{i}\rightarrow
0^{-}$ as $u\rightarrow u_{0}^{-}$ for all $i\in \left\{
k_{0}^{+}+1,...,k_{0}^{+}+l\right\} $ is - not taking the indata $f_{bu}$
into consideration - equivalent to that there will occur $l$ slowly varying
modes as $u\rightarrow u_{0}^{-}$. However, assuming that $\widetilde{\Pi }%
_{0}^{0}\left( \left( v+u\right) \mathcal{I}\left( f_{bu}\right) \right) =0$%
, will in the view of Lemma $\ref{L3}$ be equivalent to - again in the
notations of Lemma $\ref{L1}$ - that $q_{-}=%
\sum_{i=k_{0}^{+}+l+1}^{n-l}b_{i}\phi _{i}$ ($%
b_{k_{0}^{+}+1}=...=b_{k_{0}^{+}+1}=0$) also for $u<u_{0}$ as long as $l=0$
for any $u\in \lbrack u,u_{0})$. Then any smallness assumptions on $\sigma $
or $\mu $ can be removed; removing the slowly varying modes in a
neighborhood of $u_{0}$, and hence, resulting in that a uniform exponential
speed of convergence $\sigma _{u}$ in a neighborhood $\mathcal{U}$ of $u_{0}$
can be obtained. \ Indeed, in the view of Theorem $\ref{T0}$, this means to
impose $l$ extra conditions on the indata for $u$ less than $u_{0}$.

For the Boltzmann equation, the case of complete absorption at the
interface, i.e. with $R=R_{u}\equiv 0$, is well studied in the literature,
and especially, when the disturbed particles are assumed to be distributed
in accordance to the Maxwellian $M_{B}=M_{Bu}\left( \mathbf{v}\right) $ of
the interface - of the condensed phase - see e.g. the review \cite{BGS-06}
and references therein. Linearizing around a Maxwellian $M=M\left( \mathbf{v}%
\right) $%
\begin{equation*}
F=M+\sqrt{M}f
\end{equation*}%
- neglecting quadratic terms - one obtain a system of the form $\left( \ref%
{P2}\right) $, where (with $R=R_{u}\equiv 0$) 
\begin{equation*}
f_{bu}=f_{bu}\left( \mathbf{v}\right) =M^{-1/2}(M_{Bu}-M).
\end{equation*}%
Let $M_{B}=M_{B}\left( \mathbf{v}\right) $ be an arbitrary Maxwellian for a
mixture of $s$ species, and denote $M_{Bu}=M_{B}\left( \mathbf{v}+\mathbf{u}%
\right) $ after a shift $\mathbf{v\longmapsto v}+\mathbf{u}$, with $\mathbf{u%
}=\left( u,u_{2},u_{3}\right) $, in the velocity space. Then $%
f_{bu}=M^{-1/2}(M_{Bu}-M)\in \left( L^{2}\left( \left( 1+\left\vert \mathbf{v%
}\right\vert \right) \mathbf{1}_{u+v>0}\,d\mathbf{v}\right) \right) ^{s}$
has $d+s+1$ parameters: $\left\{ n_{0,1},...,n_{0,s}\right\} \subset \mathbb{%
R}_{+}$, $\mathbf{u}_{0}-\mathbf{u}\in \mathbb{R}^{d}$, $T_{0}\in \mathbb{R}%
_{+}$ - the number densities of the $s$ species, the bulk velocity (after
the shift in the velocity space), and the temperature at the boundary,
respectively. By Theorem $\ref{T1}$, if imposing $k_{u}^{+}+l_{u}$ (where
the subindex $u$ indicates the dependence on $u$) conditions on $f_{bu}$
there is a unique solution (for a fixed $u$) to problem $\left( \ref{P}%
\right) $ for each fixed $f_{bu}$ satisfying the conditions. Accordingly,
there will be (at least\footnote{%
Depending on how many of the conditions that will be of actual relevance on
the subset $\left\{ \left. f_{bu}\right\vert
\,f_{bu}=M^{-1/2}(M_{Bu}-M);\left\{ n_{0,1},...,n_{0,s}\right\} \subset 
\mathbb{R}_{+},\mathbf{u}_{0}\in \mathbb{R}^{d},T_{0}\in \mathbb{R}%
_{+}\right\} $ of $\mathcal{\mathfrak{h}}_{+}\cap \mathrm{D}(\mathcal{L})$.
Indeed, it is our firm belief that all of them will be, and that \textit{at
least} can be removed, but this still remains to be proven.}) $k_{u}^{-}$
(again the subindex $u$ indicates the dependence on $u$) free parameters of $%
f_{bu}$ left. Furthermore, by Theorem $\ref{T2}$, to have a unique solution
with a uniform exponential speed of convergence on a closed interval $%
\mathcal{U}$ of $u$, it is needed to impose $\underset{u\in \mathcal{U}}{%
\max }\left( k_{u}^{+}+l_{u}\right) $ conditions on $f_{bu}$. Then there
will be (at least\footnote{%
See comment in the footnote above.}) $\underset{u\in \mathcal{U}}{\min }%
\left( k_{u}^{-}\right) $ free parameters of $f_{bu}$ left. The
corresponding problem for monatomic single species in a neighborhood of $u=0$%
, in the nonlinear context, is considered in more details in \cite{BG-21},
see also \cite{LY-13}. It seems most likely that our results can be extended
to the weakly nonlinear case, by assuming additional conditions similar to $%
\left( \ref{d1}\right) $-$\left( \ref{d3}\right) $ on the linearized
operator, as well as some reasonable conditions on the nonlinear part,
applying methods similar to the ones in \cite{BG-21, Go-08}. Though, this
will be a topic for future studies.

\section{Appendix \label{A1}}

This appendix concerns the degenerate values of the flow velocity - the
values for which $l>0$ -, and orthogonal bases $\left( \ref{c2}\right) $ for
some important particular variants of the Boltzmann equation. In fact, the
first, second, and fourth cases below are particular cases of the fifth and
last one, but are still presented separately, due to their importance on
their own.

\paragraph{Monatomic single species}

For the Boltzmann equation for monatomic single species - in dimension $d$ -
let $\mathcal{\mathfrak{h}}=L^{2}\left( d\mathbf{v}\right) $, with the inner
product%
\begin{equation*}
\left( \left. f\right\vert g\right) =\int_{\mathbb{R}^{d}}fg\,d\mathbf{v}%
\text{, }f,g\in L^{2}\left( d\mathbf{v}\right) \text{.}
\end{equation*}%
Making the linearization%
\begin{equation}
F=M+\sqrt{M}f  \label{lin}
\end{equation}%
around a non-drifting Maxwellian 
\begin{equation*}
M=\dfrac{\rho m^{d/2}}{(2\pi T)^{d/2}}e^{-m\left\vert \mathbf{v}\right\vert
^{2}/\left( 2T\right) },
\end{equation*}%
an orthogonal basis $\left( \ref{c2}\right) $ of the kernel 
\begin{equation*}
\ker \mathcal{L}=\mathrm{span}\left\{ \sqrt{M},\sqrt{M}v,\sqrt{M}v_{2},...,%
\sqrt{M}v_{d},\sqrt{M}\left\vert \mathbf{v}\right\vert ^{2}\right\}
\end{equation*}%
of the linearized operator $\mathcal{L}$ is \cite{CGS-88}%
\begin{equation*}
\left\{ 
\begin{array}{l}
\phi _{1}=\sqrt{\frac{m}{2\rho T}}\sqrt{M}\left( \sqrt{\frac{m}{d(d+2)T}}%
\left\vert \mathbf{v}\right\vert ^{2}+v\right) \\ 
\phi _{2}=\sqrt{\frac{m}{2\rho T}}\sqrt{M}\left( \sqrt{\frac{m}{d(d+2)T}}%
\left\vert \mathbf{v}\right\vert ^{2}-v\right) \\ 
\phi _{3+j}=\sqrt{\frac{m}{\rho T}}\sqrt{M}v_{2+j},\text{ }j\in \left\{
0,...,d-2\right\} , \\ 
\phi _{d+2}=\sqrt{\frac{d+2}{2\rho }}\sqrt{M}\left( \frac{m}{(d+2)T}%
\left\vert \mathbf{v}\right\vert ^{2}-1\right) .%
\end{array}%
\right.
\end{equation*}%
Indeed, this can be obtained by - in addition to the fact that odd
components of the integrands can be discarded, due to symmetry of the
integration domain - the equalities%
\begin{eqnarray}
\left( \left. e^{-a\left\vert \mathbf{v}\right\vert ^{2}}\right\vert
1\right) &=&\left( \frac{\pi }{a}\right) ^{d/2}\text{,}  \notag \\
d\left( \left. e^{-a\left\vert \mathbf{v}\right\vert ^{2}}v\right\vert
v\right) &=&\left( \left. e^{-a\left\vert \mathbf{v}\right\vert
^{2}}\right\vert \left\vert \mathbf{v}\right\vert ^{2}\right) =\frac{d}{2a}%
\left( \frac{\pi }{a}\right) ^{d/2}\text{, and}  \notag \\
d\left( \left. e^{-a\left\vert \mathbf{v}\right\vert ^{2}}v\right\vert
v\left\vert \mathbf{v}\right\vert ^{2}\right) &=&\left( \left.
e^{-a\left\vert \mathbf{v}\right\vert ^{2}}\left\vert \mathbf{v}\right\vert
^{2}\right\vert \left\vert \mathbf{v}\right\vert ^{2}\right) =\frac{d\left(
d+2\right) }{4a^{2}}\left( \frac{\pi }{a}\right) ^{d/2}\text{, }a>0\text{.}
\label{ia1}
\end{eqnarray}%
The degenerate values of $u$ become \cite{CGS-88}%
\begin{equation*}
u_{0}=0\text{ and }u_{\pm }=\pm \sqrt{\frac{T}{m}}\sqrt{\frac{d+2}{d}}\text{,%
}
\end{equation*}%
and the values of the signature $\left( k^{+},k^{-},l\right) $ of the
restriction of the quadratic form $\left( \left. B\phi \right\vert \phi
\right) $ to the kernel of $\mathcal{L}$ depending on\ the parameter $u$ are
given by: 
\begin{equation}
\begin{tabular}[b]{|c|c|c|c|c|c|c|c|}
\hline
$u$ & $\left( -\infty ,u_{-}\right) $ & $u_{-}$ & $\left( u_{-},0\right) $ & 
$0$ & $\left( 0,u_{+}\right) $ & $u_{+}$ & $\left( u_{+},\infty \right) $ \\ 
\hline
$k^{+}$ & $0$ & $0$ & $1$ & $1$ & $d+1$ & $d+1$ & $d+2$ \\ \hline
$k^{-}$ & $d+2$ & $d+1$ & $d+1$ & $1$ & $1$ & $0$ & $0$ \\ \hline
$l$ & $0$ & $1$ & $0$ & $d$ & $0$ & $1$ & $0$ \\ \hline
\end{tabular}%
\text{.}  \label{t1}
\end{equation}

\paragraph{Multicomponent monatomic mixtures}

For the Boltzmann equation - in dimension $d$ - for a mixture of $s$ species 
$\alpha _{1},...,\alpha _{s}$, with masses $m_{\alpha _{1}},...,m_{\alpha
_{s}}$, respectively, let $\mathcal{\mathfrak{h}}=\left( L^{2}\left( d%
\mathbf{v}\right) \right) ^{s}$, with the inner product%
\begin{equation*}
\left( \left. f\right\vert g\right) =\sum_{i=1}^{s}\int_{\mathbb{R}%
^{d}}f_{i}g_{i}\,d\mathbf{v}\text{, }f,g\in \left( L^{2}\left( d\mathbf{v}%
\right) \right) ^{s}\text{.}
\end{equation*}%
Making the linearization $\left( \ref{lin}\right) $ around a non-drifting
Maxwellian 
\begin{equation*}
M=\left( M_{\alpha _{1}},...,M_{\alpha _{s}}\right) ,\text{ }M_{\alpha
_{i}}=n_{\alpha _{i}}\left( \dfrac{m_{\alpha _{i}}}{2\pi T}\right)
^{d/2}e^{-m_{\alpha _{i}}\left\vert \mathbf{v}\right\vert ^{2}/\left(
2T\right) },
\end{equation*}%
an orthogonal basis $\left( \ref{c2}\right) $ of the kernel 
\begin{eqnarray*}
\ker \mathcal{L} &=&\mathrm{span}\left\{ \sqrt{M_{\alpha _{1}}}\mathbf{e}%
_{1},...,\sqrt{M_{\alpha _{s}}}\mathbf{e}_{s},\sqrt{\overline{M}}v,\sqrt{%
\overline{M}}v_{2},\sqrt{\overline{M}}v_{3},\sqrt{\overline{M}}\left\vert 
\mathbf{v}\right\vert ^{2}\right\} , \\
\overline{M} &=&\left( m_{\alpha _{1}}^{2}M_{\alpha _{1}},...,m_{\alpha
_{s}}^{2}M_{\alpha _{s}}\right) \text{, }\mathbf{e}_{i}=(\underset{i-1}{%
\underbrace{0,...,0}},1\underset{s-i}{,\underbrace{0,...,0}})\text{ for }%
i\in \left\{ 1,...,s\right\}
\end{eqnarray*}%
of the linearized operator $\mathcal{L}$ is%
\begin{align*}
& \left\{ 
\begin{array}{l}
\phi _{1,\alpha _{i}}=\frac{m_{\alpha _{i}}}{\sqrt{2\rho T}}\sqrt{M_{\alpha
_{i}}}\left( \sqrt{\frac{\rho }{d(d+2)nT}}\left\vert \mathbf{v}\right\vert
^{2}+v\right) \\ 
\phi _{2,\alpha _{i}}=\frac{m_{\alpha _{i}}}{\sqrt{2\rho T}}\sqrt{M_{\alpha
_{i}}}\left( \sqrt{\frac{\rho }{d(d+2)nT}}\left\vert \mathbf{v}\right\vert
^{2}-v\right) \\ 
\phi _{3+j,\alpha _{i}}=\sqrt{\frac{m_{\alpha _{i}}}{nT}}\sqrt{M_{\alpha
_{i}}}v_{2+j},\;j\in \left\{ 0,...,d-2\right\} , \\ 
\phi _{d+1+j}=\dfrac{\widetilde{\phi }_{d+1+j}}{\left\Vert \widetilde{\phi }%
_{d+1+j}\right\Vert },\;j\in \left\{ 1,...,s\right\} ,%
\end{array}%
\right. \\
\widetilde{\phi }_{d+1+j}& =\widehat{\phi }_{d+1+j}-\sum_{k=2}^{j}\left(
\left. \widehat{\phi }_{d+1+j}\right\vert \phi _{d+k}\right) \phi _{d+k}, \\
\widehat{\phi }_{d+1+j,\alpha _{i}}& =\sqrt{M_{\alpha _{i}}}\left( \frac{%
n_{\alpha _{j}}m_{\alpha _{i}}}{(d+2)nT}\left\vert \mathbf{v}\right\vert
^{2}-\delta _{ij}\right) ,\text{ }\rho =\sum_{k=1}^{s}m_{\alpha
_{k}}n_{\alpha _{k}},\;n=\sum_{k=1}^{s}n_{\alpha _{k}}.
\end{align*}%
Again this can be obtained by equalities $\left( \ref{ia1}\right) $. The
degenerate values of $u$ become \cite{Be-17} 
\begin{equation*}
u_{0}=0\text{ and }u_{\pm }=\pm \sqrt{\frac{nT}{\rho }}\sqrt{\frac{d+2}{d}}.
\end{equation*}%
and the values of the signature $\left( k^{+},k^{-},l\right) $ of the
restriction of the quadratic form $\left( \left. B\phi \right\vert \phi
\right) $ to the kernel of $\mathcal{L}$ depending on\ the parameter $u$ are
given by:%
\begin{equation}
\begin{tabular}[b]{|c|c|c|c|c|c|c|c|}
\hline
$u$ & $\left( -\infty ,u_{-}\right) $ & $u_{-}$ & $\left( u_{-},0\right) $ & 
$0$ & $\left( 0,u_{+}\right) $ & $u_{+}$ & $\left( u_{+},\infty \right) $ \\ 
\hline
$k^{+}$ & $0$ & $0$ & $1$ & $1$ & $d+s$ & $d+s$ & $d+s+1$ \\ \hline
$k^{-}$ & $d+s+1$ & $d+s$ & $d+s$ & $1$ & $1$ & $0$ & $0$ \\ \hline
$l$ & $0$ & $1$ & $0$ & $d+s-1$ & $0$ & $1$ & $0$ \\ \hline
\end{tabular}%
.  \label{t2}
\end{equation}

\paragraph{Nordheim-Boltzmann equation}

For the Nordheim-Boltzmann equation for monatomic single species - in
dimension $d$ (with $d\geq 2$) - let $\mathcal{\mathfrak{h}}=L^{2}\left( d%
\mathbf{p}\right) $, with the inner product%
\begin{equation*}
\left( \left. f\right\vert g\right) =\int_{\mathbb{R}^{d}}fg\,d\mathbf{p}%
\text{, }f,g\in L^{2}\left( d\mathbf{p}\right) \text{,}
\end{equation*}%
and linearize by%
\begin{equation*}
F=P_{\varepsilon }+(P_{\varepsilon }(1+\varepsilon P_{\varepsilon }))^{1/2}f,
\end{equation*}%
with $\varepsilon =1$ for bosons, and $\varepsilon =-1$ for fermions, around
a non-drifting Planckian 
\begin{equation*}
P_{\varepsilon }=P_{\pm }=\dfrac{1}{e^{\left\vert \mathbf{p}\right\vert
^{2}/\left( 2T\right) }\mp 1},\text{ }\mathbf{p}=\left(
p_{1},...,p_{d}\right) .
\end{equation*}%
For bosons we consider the restriction $\left\vert \mathbf{p}\right\vert
\geq \lambda \sqrt{2T}$, for some $\lambda >0$, cf. \cite{AN-13, Be-15} (to
avoid a singularity at $\mathbf{p=0}$). Then an orthogonal basis $\left( \ref%
{c2}\right) $ of the kernel 
\begin{eqnarray*}
\ker \mathcal{L} &=&\mathrm{span}\left\{ \sqrt{\mathcal{R}_{\pm }},\sqrt{%
\mathcal{R}_{\pm }}p_{1},...,\sqrt{\mathcal{R}_{\pm }}p_{d},\sqrt{\mathcal{R}%
_{\pm }}\left\vert \mathbf{p}\right\vert ^{2}\right\} ,\text{ with } \\
\mathcal{R}_{\pm } &=&P_{\pm }(1\pm P_{\pm })\text{,}
\end{eqnarray*}%
of the linearized operator $\mathcal{L}$ is%
\begin{equation*}
\left\{ 
\begin{array}{l}
\phi _{1}^{\pm }=\frac{1}{\left( 2T\right) ^{(d+2)/4}\pi ^{d/4}}\sqrt{P_{\pm
}(1\pm P_{\pm })}\left( \sqrt{\frac{1}{J_{d+2}^{\pm }d\left( d+2\right) T}}%
\left\vert \mathbf{p}\right\vert ^{2}+\sqrt{\frac{1}{J_{d}^{\pm }}}%
p_{1}\right) \\ 
\phi _{2}^{\pm }=\frac{1}{\left( 2T\right) ^{(d+2)/4}\pi ^{d/4}}\sqrt{P_{\pm
}(1\pm P_{\pm })}\left( \sqrt{\frac{1}{J_{d+2}^{\pm }d\left( d+2\right) T}}%
\left\vert \mathbf{p}\right\vert ^{2}-\sqrt{\frac{1}{J_{d}^{\pm }}}%
p_{1}\right) \\ 
\phi _{3+j}^{\pm }=\frac{1}{\left( 2T\right) ^{(d+2)/4}\pi ^{d/4}}\sqrt{%
\frac{2}{J_{d}^{\pm }}}\sqrt{P_{\pm }(1\pm P_{\pm })}p_{2+j}\text{,}\;j\in
\left\{ 0,...,d-2\right\} , \\ 
\phi _{d+2}^{\pm }=\frac{1}{\left( 2\pi T\right) ^{d/4}}\sqrt{\frac{%
(d+2)J_{d+2}^{\pm }}{(d+2)J_{d-2}^{\pm }J_{d+2}^{\pm }-d\left( J_{d}^{\pm
}\right) ^{2}}}\sqrt{P_{\pm }(1\pm P_{\pm })}\left( \frac{J_{d}^{\pm }}{%
J_{d+2}^{\pm }\left( d+2\right) T}\left\vert \mathbf{p}\right\vert
^{2}-1\right) ,%
\end{array}%
\right.
\end{equation*}%
with 
\begin{equation*}
J_{s}^{-}=\dfrac{2}{\Gamma \left( s/2+1\right) }\int_{0}^{\infty }\dfrac{%
r^{s+1}e^{r^{2}}}{\left( e^{r^{2}}+1\right) ^{2}}\,dr=\eta \left( s/2\right)
\ \text{for}\;s\geq 0,
\end{equation*}%
where $\Gamma $ is the gamma-function and $\eta =\dfrac{1}{\Gamma \left(
s\right) }\int\limits_{0}^{\infty }\dfrac{r^{s-1}}{e^{r}+1}\,dr$ is the
Dirichlet eta-function (alternating zeta-function), for fermions; while, for$%
\;s\geq 0,$ 
\begin{eqnarray*}
J_{s}^{+} &=&\dfrac{2}{\Gamma \left( s/2+1\right) }\int_{\lambda }^{\infty }%
\dfrac{r^{s+1}e^{r^{2}}}{\left( e^{r^{2}}-1\right) ^{2}}\,dr \\
&=&\dfrac{1}{\Gamma \left( s/2+1\right) }\dfrac{\lambda ^{s}}{e^{\lambda
^{2}}-1}+\dfrac{1}{\Gamma \left( s/2\right) }\int_{\lambda ^{2}}^{\infty }%
\dfrac{r^{s/2-1}}{e^{r}-1}\,dr\ 
\end{eqnarray*}%
for bosons. Indeed, this can be obtained by - in addition to the fact that
odd components of the integrands can be discarded, due to symmetry of the
integration domain - the equalities%
\begin{eqnarray*}
\left( \left. \frac{e^{a\left\vert \mathbf{p}\right\vert ^{2}}}{\left(
e^{a\left\vert \mathbf{p}\right\vert ^{2}}\mp 1\right) ^{2}}\right\vert
1\right) &=&S_{d-1}\frac{\Gamma \left( d/2\right) }{2a^{d/2}}J_{d-2}^{\pm
}=\left( \frac{\pi }{a}\right) ^{d/2}J_{d-2}^{\pm }\text{,} \\
d\left( \left. \frac{e^{a\left\vert \mathbf{p}\right\vert ^{2}}}{\left(
e^{a\left\vert \mathbf{p}\right\vert ^{2}}\mp 1\right) ^{2}}p_{1}\right\vert
p_{1}\right) &=&\left( \left. \frac{e^{a\left\vert \mathbf{p}\right\vert
^{2}}}{\left( e^{a\left\vert \mathbf{p}\right\vert ^{2}}\mp 1\right) ^{2}}%
\right\vert \left\vert \mathbf{p}\right\vert ^{2}\right) \\
&=&S_{d-1}\frac{\Gamma \left( d/2+1\right) }{2a^{\left( d+2\right) /2}}%
J_{d}^{\pm }=\frac{d}{2a}\left( \frac{\pi }{a}\right) ^{d/2}J_{d}^{\pm }%
\text{, and} \\
d\left( \left. \frac{e^{a\left\vert \mathbf{p}\right\vert ^{2}}}{\left(
e^{a\left\vert \mathbf{p}\right\vert ^{2}}\mp 1\right) ^{2}}p_{1}\right\vert
p_{1}\left\vert \mathbf{p}\right\vert ^{2}\right) &=&\left( \left. \frac{%
e^{a\left\vert \mathbf{p}\right\vert ^{2}}}{\left( e^{a\left\vert \mathbf{p}%
\right\vert ^{2}}\mp 1\right) ^{2}}\left\vert \mathbf{p}\right\vert
^{2}\right\vert \left\vert \mathbf{p}\right\vert ^{2}\right) \\
&=&S_{d-1}\frac{\Gamma \left( d/2+2\right) }{2a^{\left( d+4\right) /2}}%
J_{d+2}^{\pm } \\
&=&\frac{d\left( d+2\right) }{4a^{2}}\left( \frac{\pi }{a}\right)
^{d/2}J_{d+2}^{\pm }\text{ for }a>0\text{, } \\
\text{while }\Gamma \left( s+1\right) &=&s\Gamma \left( s\right) \text{ for }%
s>0\text{,}
\end{eqnarray*}%
where $S_{n}$ is the surface area of the $n$-sphere $\mathbb{S}^{n}$, i.e. $%
S_{n-1}=\dfrac{2\pi ^{n/2}}{\Gamma \left( n/2\right) }$.

Then the degenerate values of $u$ become \cite{Be-17}%
\begin{equation}
u_{0}=0\text{ and }u_{\pm }=\pm \sqrt{\frac{\eta \left( d/2+1\right) }{\eta
\left( d/2\right) }}\sqrt{T}\sqrt{\frac{d+2}{d}}  \label{m1}
\end{equation}%
for fermions, and 
\begin{equation}
u_{0}=0\text{ and }u_{\pm }=\sqrt{\frac{J_{d+2}^{+}}{J_{d}^{+}}}\sqrt{T}%
\sqrt{\frac{d+2}{d}}  \label{m2}
\end{equation}%
for bosons. Note that, for $0\leq s\leq 2$, $J_{s}^{+}\rightarrow \infty $
as $\lambda $$\rightarrow 0$, while, for $s\geq 2$, $J_{s}^{+}\rightarrow
\zeta \left( s/2\right) $ as $\lambda $$\rightarrow 0$, where $\zeta =\dfrac{%
1}{\Gamma \left( s\right) }\int\limits_{0}^{\infty }\dfrac{r^{s-1}}{e^{r}-1}%
\,dr$ is the zeta-function (note that $\zeta \left( 1\right) $ is infinite);
and hence, $u_{\pm }$ $\rightarrow \pm \sqrt{\dfrac{\zeta \left(
d/2+1\right) }{\zeta \left( d/2\right) }}\sqrt{T}\sqrt{\dfrac{d+2}{d}}$ as $%
\lambda $$\rightarrow 0$. The values of the signature $\left(
k^{+},k^{-},l\right) $ of the restriction of the quadratic form $\left(
\left. B\phi \right\vert \phi \right) $ to the kernel of $\mathcal{L}$
depending on\ the parameter $u$ are given by table $\left( \ref{t1}\right) $%
, with $u_{\pm }$ in $\left( \ref{m1}\right) $ and $\left( \ref{m2}\right) $%
, respectively.

\paragraph{Single species of polyatomic molecules}

i) For the Boltzmann equation - in dimension $d$ - for a polyatomic single
species with $r$ different internal energy levels $E_{1},...,E_{r}$, let $%
\mathcal{\mathfrak{h}}=\left( L^{2}\left( d\mathbf{v}\right) \right) ^{r}$,
with the inner product%
\begin{equation*}
\left( \left. f\right\vert g\right) =\sum\limits_{i=1}^{r}\int_{\mathbb{R}%
^{d}}f_{i}g_{i}\,d\mathbf{v}\text{, }f,g\in \left( L^{2}\left( d\mathbf{v}%
\right) \right) ^{r}\text{.}
\end{equation*}%
Making the linearization $\left( \ref{lin}\right) $ around a non-drifting
Maxwellian 
\begin{eqnarray*}
M &=&\left( M_{1},...,M_{r}\right) ,\text{ where }M_{i}=\dfrac{\rho \varphi
_{i}m^{d/2}}{\left( 2\pi T\right) ^{d/2}Q_{0}}e^{-\left( m\left\vert \mathbf{%
v}\right\vert ^{2}+2E_{i}\right) /\left( 2T\right) },\text{ with} \\
Q_{j} &=&\sum_{i=1}^{r}\varphi _{i}E_{i}^{j}e^{-E_{i}/T}\text{ for}\;j\in
\left\{ 0,1,2\right\} ,
\end{eqnarray*}%
where $\varphi _{i}=\varphi (E_{i})$ for $i\in \left\{ 1,...,r\right\} $ are
given weights; an orthogonal basis $\left( \ref{c2}\right) $ of the kernel 
\begin{eqnarray*}
\ker \mathcal{L} &=&\mathrm{span}\left\{ \sqrt{M},\sqrt{M}v,\sqrt{M}%
v_{2},...,\sqrt{M}v_{d},\sqrt{M}\left\vert \mathbf{v}\right\vert ^{2}+2%
\mathcal{E}_{M}\sqrt{M}\right\} , \\
\mathcal{E}_{M} &=&\mathrm{diag}\left( E_{1},...,E_{r}\right) \text{,}
\end{eqnarray*}%
of the linearized operator $\mathcal{L}$ is 
\begin{equation*}
\left\{ 
\begin{array}{l}
\phi _{1,i}=\frac{1}{\sqrt{2\rho T}}\sqrt{M_{i}}\left( \Psi _{i}+\sqrt{m}%
v\right) \\ 
\phi _{2,i}=\frac{1}{\sqrt{2\rho T}}\sqrt{M_{i}}\left( \Psi _{i}-\sqrt{m}%
v\right) \\ 
\phi _{3+j,i}=\sqrt{\frac{m}{\rho T}}\sqrt{M_{i}}v_{2+j}\text{,}\;j\in
\left\{ 0,...,d-2\right\} , \\ 
\phi _{d+2,i}=\sqrt{\frac{d+2+\kappa }{2\rho }}\sqrt{M_{i}}\left( \sqrt{%
\frac{\left( d+\kappa \right) }{\left( d+2+\kappa \right) T}}\Psi
_{i}-1\right) .%
\end{array}%
\right.
\end{equation*}%
with%
\begin{eqnarray}
\Psi _{i} &=&\frac{1}{\sqrt{\left( d+\kappa \right) \left( d+2+\kappa
\right) T}}\left( m\left\vert \mathbf{v}\right\vert ^{2}+2E_{i}+\kappa T-2%
\frac{Q_{1}}{Q_{0}}\right)  \notag \\
\kappa &=&\frac{2}{T^{2}}\frac{Q_{0}Q_{2}-Q_{1}^{2}}{Q_{0}^{2}}.  \label{m3a}
\end{eqnarray}%
This can be obtained by equalities $\left( \ref{ia1}\right) $, indeed, it
follows that 
\begin{eqnarray}
\left( \left. \sqrt{M}\right\vert \sqrt{M}\right) &=&\rho \text{,}  \notag \\
\left( \left. \sqrt{M}v\right\vert \sqrt{M}v\right) &=&\left( \left. \sqrt{M}%
v_{i}\right\vert \sqrt{M}v_{i}\right) =\frac{\rho }{m}T\text{, }i\in \left\{
2,...,d\right\} \text{,}  \notag \\
\left( \left. \sqrt{M}\right\vert \mathfrak{E}_{M}\right) &=&\frac{\rho }{m}%
\left( dT+\frac{2Q_{1}}{Q_{0}}\right) \text{,}  \notag \\
\left( \left. \sqrt{M}v^{2}\right\vert \mathfrak{E}_{M}\right) &=&\frac{\rho 
}{m^{2}}\left( \left( d+2\right) T^{2}+2T\frac{Q_{1}}{Q_{0}}\right) \text{,
and}  \notag \\
\left( \left. \mathfrak{E}_{M}\right\vert \mathfrak{E}_{M}\right) &=&\frac{%
\rho }{m^{2}}\left( d\left( d+2\right) T^{2}+4dT\frac{Q_{1}}{Q_{0}}+4\frac{%
Q_{2}}{Q_{0}}\right) \text{,}  \notag \\
\text{with }\mathfrak{E}_{M} &=&\sqrt{M}\left\vert \mathbf{v}\right\vert
^{2}+2\mathcal{E}_{M}\sqrt{M}  \label{ia2}
\end{eqnarray}%
Then the degenerate values of $u$ become \cite{Be-16b} 
\begin{equation}
u_{0}=0\text{ and }u_{\pm }=\pm \sqrt{\frac{T}{m}}\sqrt{\frac{d+2+\kappa }{%
d+\kappa }},  \label{m3}
\end{equation}%
and the values of the signature $\left( k^{+},k^{-},l\right) $ of the
restriction of the quadratic form $\left( \left. B\phi \right\vert \phi
\right) $ to the kernel of $\mathcal{L}$ depending on\ the parameter $u$ are
given by table $\left( \ref{t1}\right) $, with $u_{\pm }$ in $\left( \ref{m3}%
\right) $.

ii) For the Boltzmann equation - in dimension $d$ - for a polyatomic single
species, where the polyatomicity is modelled by a continuous internal energy
variable $I$, let $\mathcal{\mathfrak{h}}=L^{2}\left( d\mathbf{v}dI\right) $%
, with the inner product%
\begin{equation*}
\left( \left. f\right\vert g\right) =\int_{0}^{1}\int_{\mathbb{R}%
^{d}}f_{i}g_{i}\,d\mathbf{v}dI\text{, }f,g\in L^{2}\left( d\mathbf{v}%
dI\right) \text{.}
\end{equation*}%
Making the linearization $\left( \ref{lin}\right) $ around a non-drifting
Maxwellian 
\begin{equation*}
M=\dfrac{\rho \varphi (I)m^{d/2}}{\left( 2\pi T\right) ^{d/2}Q}e^{-\left(
m\left\vert \mathbf{v}\right\vert ^{2}+2I\right) /\left( 2T\right) },\text{
with }Q=\int_{0}^{\infty }\varphi (I)e^{-I/T}\,dI\text{,}
\end{equation*}%
with $\varphi (I)=I^{\delta /2-1}$ - where $\delta $ is the number of
internal degrees of freedom; an orthogonal basis $\left( \ref{c2}\right) $
of the kernel 
\begin{equation*}
\ker \mathcal{L}=\mathrm{span}\left\{ \sqrt{M},\sqrt{M}v,\sqrt{M}v_{2},...,%
\sqrt{M}v_{d},\sqrt{M}\left( m\left\vert \mathbf{v}\right\vert
^{2}+2I\right) \right\}
\end{equation*}%
of the linearized operator $\mathcal{L}$ is 
\begin{equation*}
\left\{ 
\begin{array}{l}
\phi _{1}=\frac{\sqrt{m}}{\sqrt{2\rho T}}\sqrt{M}\left( \frac{\sqrt{m}}{%
\sqrt{\left( d+\delta \right) \left( d+2+\delta \right) T}}\left( \left\vert 
\mathbf{v}\right\vert ^{2}+\frac{2I}{m}\right) +v\right) \\ 
\phi _{2}=\frac{\sqrt{m}}{\sqrt{2\rho T}}\sqrt{M}\left( \frac{\sqrt{m}}{%
\sqrt{\left( d+\delta \right) \left( d+2+\delta \right) T}}\left( \left\vert 
\mathbf{v}\right\vert ^{2}+\frac{2I}{m}\right) -v\right) \\ 
\phi _{3+j}=\frac{\sqrt{m}}{\sqrt{\rho T}}\sqrt{M}v_{2+j}\text{,}\;j\in
\left\{ 0,...,d-2\right\} , \\ 
\phi _{d+2}=\frac{\sqrt{d+2+\delta }}{\sqrt{2\rho }}\sqrt{M}\left( \frac{m}{%
\left( d+2+\delta \right) T}\left( \left\vert \mathbf{v}\right\vert ^{2}+%
\frac{2I}{m}\right) -1\right) .%
\end{array}%
\right.
\end{equation*}%
This can be obtained by observing that, cf. $\left( \ref{m3a}\right) $,%
\begin{equation}
\frac{2}{T^{2}}\frac{Q\int_{0}^{\infty }\varphi (I)I^{2}e^{-I/T}\,dI-\left(
\int_{0}^{\infty }\varphi (I)Ie^{-I/T}\,dI\right) ^{2}}{Q^{2}}=\delta
\label{m4a}
\end{equation}%
since%
\begin{equation*}
\int_{0}^{\infty }\varphi (I)I^{2}e^{-I/T}\,dI=\frac{T\left( \delta
+2\right) }{2}\int_{0}^{\infty }\varphi (I)Ie^{-I/T}\,dI=\frac{T^{2}\left(
\delta +2\right) \delta }{4}Q\text{.}
\end{equation*}%
Then the degenerate values of $u$ become \cite{Be-16b} \ 
\begin{equation}
u_{0}=0\text{ and }u_{\pm }=\pm \sqrt{\frac{T}{m}}\sqrt{\frac{d+2+\delta }{%
d+\delta }}.  \label{m4}
\end{equation}%
and the values of the signature $\left( k^{+},k^{-},l\right) $ of the
restriction of the quadratic form $\left( \left. B\phi \right\vert \phi
\right) $ to the kernel of $\mathcal{L}$ depending on\ the parameter $u$ are
given by table $\left( \ref{t1}\right) $, with $u_{\pm }$ in $\left( \ref{m4}%
\right) $.

\paragraph{Multicomponent mixtures of polyatomic molecules}

i) For the Boltzmann equation - in dimension $d$ - for a mixture of
polyatomic molecules, consisting of $s$ different species $\alpha
_{1},...,\alpha _{s}$, with masses $m_{\alpha _{1}},...,m_{\alpha _{s}}$,
and with, for each species $\alpha _{i}$, $r_{i}$ different internal energy
levels $E_{1}^{\alpha _{i}},...,E_{r_{i}}^{\alpha _{i}}$, let $\mathcal{%
\mathfrak{h}}=\left( L^{2}\left( d\mathbf{v}\right) \right) ^{\widehat{r}}$,
where $\widehat{r}=\sum_{i=1}^{s}r_{i}$, with the inner product%
\begin{equation*}
\left( \left. f\right\vert g\right)
=\sum_{i=1}^{s}\sum\limits_{j=1}^{r_{i}}\int_{\mathbb{R}^{d}}f_{i,j}g_{i,j}%
\,d\mathbf{v}\text{, }f,g\in \left( L^{2}\left( d\mathbf{v}\right) \right) ^{%
\widehat{r}}\text{, }\widehat{r}=\sum_{i=1}^{s}r_{i}\text{.}
\end{equation*}%
Making the linearization $\left( \ref{lin}\right) $ around a non-drifting
Maxwellian 
\begin{eqnarray*}
M &=&\left( M_{\alpha _{1}},...,M_{\alpha _{s}}\right) ,\;M_{\alpha
_{i}}=\left( M_{\alpha _{i},1},...,M_{\alpha _{i},r_{i}}\right) , \\
M_{\alpha _{i},j} &=&\dfrac{n_{\alpha _{i}}\varphi _{\alpha _{i},j}m_{\alpha
_{i}}^{d/2}}{\left( 2\pi T\right) ^{d/2}Q_{0}^{\alpha _{i}}}e^{-\left(
m_{\alpha _{i}}\left\vert \mathbf{v}\right\vert ^{2}+2E_{j}^{\alpha
_{i}}\right) /\left( 2T\right) }, \\
Q_{k}^{\alpha _{i}} &=&\sum\limits_{j=1}^{r_{i}}\varphi _{\alpha
_{i},j}\left( E_{j}^{\alpha _{i}}\right) ^{k}e^{-E_{j}^{\alpha
_{i}}/T},\;k\in \left\{ 0,1,2\right\} ,
\end{eqnarray*}%
while $\varphi _{\alpha _{i},j_{i}}=\varphi _{\alpha _{i}}\left(
E_{j_{i}}^{\alpha _{i}}\right) $, $j_{i}\in \left\{ 1,...,r_{i}\right\} $, $%
i\in \left\{ 1,...,s\right\} $, are given weights; an orthogonal basis $%
\left( \ref{c2}\right) $ of the kernel 
\begin{eqnarray*}
\ker \mathcal{L} &=&\mathrm{span}\left\{ \sqrt{\mathcal{M}_{\alpha _{1}}}%
,...,\sqrt{\mathcal{M}_{\alpha _{s}}},\mathfrak{m}\sqrt{M}v,\mathfrak{m}%
\sqrt{M}v_{2},...,\mathfrak{m}\sqrt{M}v_{d},\mathfrak{E}_{M}\right\} , \\
\mathfrak{m} &=&\mathrm{diag}(\underset{r_{1}}{\underbrace{m_{\alpha
_{1}},...,m_{\alpha _{1}}}},...,\underset{r_{s}}{\underbrace{m_{\alpha
_{s}},...,m_{\alpha _{s}}}})\text{, }\mathcal{M}_{\alpha _{i}}=\left(
0,...,0,M_{\alpha _{i}},0,...,0\right) \text{,} \\
\mathfrak{E}_{M} &=&\left( \mathfrak{m}\left\vert \mathbf{v}\right\vert
^{2}+2\mathcal{E}_{M}\right) \sqrt{M}\text{, }\mathcal{E}_{M}=\mathrm{diag}%
\left( E_{1}^{\alpha _{1}},...,E_{r_{1}}^{\alpha _{1}},...,E_{1}^{\alpha
_{s}},...,E_{r_{s}}^{\alpha _{s}}\right) ,
\end{eqnarray*}%
of the linearized operator $\mathcal{L}$ is%
\begin{equation*}
\left\{ 
\begin{array}{l}
\phi _{1,\alpha _{i}j}=\frac{1}{\sqrt{2\rho T}}\sqrt{M_{\alpha _{i},j}}%
\left( \Psi _{\alpha _{i}j}+m_{\alpha _{i}}v\right) \\ 
\phi _{2,\alpha _{i}j}=\frac{1}{\sqrt{2\rho T}}\sqrt{M_{\alpha _{i},j}}%
\left( \Psi _{\alpha _{i}j}-m_{\alpha _{i}}v\right) \\ 
\phi _{3+k,\alpha _{i},j}=\sqrt{\frac{m_{\alpha _{i}}}{nT}}\sqrt{M_{\alpha
_{i},j}}v_{2+k},\;k\in \left\{ 0,...,d-2\right\} , \\ 
\phi _{d+1+k}=\dfrac{\widetilde{\phi }_{d+1+k}}{\left\Vert \widetilde{\phi }%
_{d+1+k}\right\Vert }\text{,}\;k\in \left\{ 1,...,s\right\} ,%
\end{array}%
\right.
\end{equation*}%
with%
\begin{align*}
\Psi _{\alpha _{i}j}& =\sqrt{\frac{\rho }{\left( d+\overline{\kappa }\right)
\left( d+2+\overline{\kappa }\right) nT}}\left( m_{\alpha _{i}}\left\vert 
\mathbf{v}\right\vert ^{2}+2E_{j}^{\alpha _{i}}+\overline{\kappa }T-2\frac{%
Q_{1}^{\alpha _{i}}}{Q_{0}^{\alpha _{i}}}\right) \\
\widetilde{\phi }_{d+1+k}& =\widehat{\phi }_{d+1+k}-\sum_{p=2}^{j}\left(
\left. \widehat{\phi }_{d+1+k}\right\vert \phi _{d+p}\right) \phi _{d+k}, \\
\widehat{\phi }_{d+1+k,\alpha _{i},j}& =\sqrt{M_{\alpha _{i},j}}\left( \sqrt{%
\frac{d+\overline{\kappa }}{\left( d+2+\overline{\kappa }\right) \rho nT}}%
n_{\alpha _{k}}\Psi _{\alpha _{i}j}-\delta _{ik}\right) , \\
\rho & =\sum_{i=1}^{s}m_{\alpha _{i}}n_{\alpha
i},\;n=\sum_{k=1}^{s}n_{\alpha _{k}},\;\overline{\kappa }=\frac{2}{nT^{2}}%
\sum_{i=1}^{s}\frac{Q_{0}^{\alpha _{i}}Q_{2}^{\alpha _{i}}-\left(
Q_{1}^{\alpha _{i}}\right) ^{2}}{\left( Q_{0}^{\alpha _{i}}\right) ^{2}}%
n_{\alpha _{i}}.
\end{align*}%
Again this can be obtained by equalities $\left( \ref{ia1}\right) $, cf.
also equalities $\left( \ref{ia2}\right) $. Then the degenerate values of $u$
become 
\begin{equation}
u_{0}=0\text{ and }u_{\pm }=\pm \sqrt{\frac{nT}{\rho }}\sqrt{\frac{d+2+%
\overline{\kappa }}{d+\overline{\kappa }}}\text{.}  \label{m5}
\end{equation}%
and the values of the signature $\left( k^{+},k^{-},l\right) $ of the
restriction of the quadratic form $\left( \left. B\phi \right\vert \phi
\right) $ to the kernel of $\mathcal{L}$ depending on\ the parameter $u$ are
given by table $\left( \ref{t2}\right) $, with $u_{\pm }$ in $\left( \ref{m5}%
\right) $.

ii) For the Boltzmann equation - in dimension $d$ - for a polyatomic
multicomponent mixture of $s$ species with masses $m_{\alpha
_{1}},...,m_{\alpha _{s}}$, respectively, where the polyatomicity is
modelled by a continuous internal energy variable $I$, let $\mathcal{%
\mathfrak{h}}=\left( L^{2}\left( d\mathbf{v}dI\right) \right) ^{s}$, with
the inner product%
\begin{equation*}
\left( \left. f\right\vert g\right) =\sum_{i=1}^{s}\int_{0}^{1}\int_{\mathbb{%
R}^{d}}f_{i}g_{i}\,d\mathbf{v}dI\text{, }f,g\in \left( L^{2}\left( d\mathbf{v%
}dI\right) \right) ^{s}\text{.}
\end{equation*}%
Making the linearization $\left( \ref{lin}\right) $ around a non-drifting
Maxwellian 
\begin{eqnarray*}
M &=&\left( M_{\alpha _{1}},...,M_{\alpha _{s}}\right) \text{,}\;M_{\alpha
_{i}}=\dfrac{n_{\alpha _{i}}\varphi _{\alpha _{i}}(I)m_{\alpha _{i}}^{d/2}}{%
\left( 2\pi T\right) ^{d/2}Q_{\alpha _{i}}}e^{-\left( m_{\alpha
_{i}}\left\vert \mathbf{v}\right\vert ^{2}+2I\right) /\left( 2T\right) }, \\
Q_{\alpha _{i}} &=&\int_{0}^{\infty }\varphi _{\alpha _{i}}(I)e^{-I/T}\,dI%
\text{,}
\end{eqnarray*}%
with $\varphi _{\alpha _{i}}(I)=I^{\delta _{\alpha _{i}}/2-1}$, where $%
\delta _{\alpha _{i}}$ is the number of internal degrees of freedom for
species $\alpha _{i}$; an orthogonal basis $\left( \ref{c2}\right) $ of the
kernel 
\begin{eqnarray*}
\ker \mathcal{L} &=&\mathrm{span}\left\{ \sqrt{M_{\alpha _{1}}}\mathbf{e}%
_{1},...,\sqrt{M_{\alpha _{s}}}\mathbf{e}_{s},\mathfrak{m}\sqrt{M}v,%
\mathfrak{m}\sqrt{M}v_{2},...,\mathfrak{m}\sqrt{M}v_{d},\mathfrak{E}%
_{M}\right\} , \\
&&\text{with }\mathfrak{m}=\mathrm{diag}\left( m_{\alpha _{1}},...,m_{\alpha
_{s}}\right) \text{ and }\mathfrak{E}_{M}=\mathfrak{m}\sqrt{M}\left\vert 
\mathbf{v}\right\vert ^{2}+2I\sqrt{M}\text{,}
\end{eqnarray*}%
of the linearized operator $\mathcal{L}$ is%
\begin{equation*}
\left\{ 
\begin{array}{l}
\phi _{1,\alpha _{i}}=\frac{1}{\sqrt{2\rho T}}\sqrt{M_{\alpha _{i}}}\left(
\Psi _{\alpha _{i}}+m_{\alpha _{i}}v\right) \\ 
\phi _{2,\alpha _{i}}=\frac{1}{\sqrt{2\rho T}}\sqrt{M_{\alpha _{i}}}\left(
\Psi _{\alpha _{i}}-m_{\alpha _{i}}v\right) \\ 
\phi _{3+j,\alpha _{i}}=\sqrt{\frac{m_{\alpha _{i}}}{nT}}\sqrt{M_{\alpha
_{i}}}v_{2+j},\;j\in \left\{ 0,...,d-2\right\} , \\ 
\phi _{d+1+j}=\dfrac{\widetilde{\phi }_{d+1+j}}{\left\Vert \widetilde{\phi }%
_{d+1+j}\right\Vert }\text{,}\;j\in \left\{ 1,...,s\right\} ,%
\end{array}%
\right.
\end{equation*}%
with%
\begin{align*}
\Psi _{\alpha _{i}}& =\sqrt{\frac{\rho }{\left( d+\overline{\delta }\right)
\left( d+2+\overline{\delta }\right) nT}}\left( m_{\alpha _{i}}\left\vert 
\mathbf{v}\right\vert ^{2}+2I+\left( \overline{\delta }-\delta _{\alpha
_{i}}\right) T\right) \\
\widetilde{\phi }_{d+1+j}& =\widehat{\phi }_{d+1+j}-\sum_{k=2}^{j}\left(
\left. \widehat{\phi }_{d+1+j}\right\vert \phi _{d+k}\right) \phi _{d+k}, \\
\widehat{\phi }_{d+1+j,\alpha _{i}}& =\sqrt{M_{\alpha _{i}}}\left( \sqrt{%
\frac{d+\overline{\delta }}{(d+2+\overline{\delta })\rho nT}}n_{\alpha
_{j}}\Psi _{\alpha _{i}}-\delta _{ij}\right) , \\
\rho & =\sum_{k=1}^{s}m_{\alpha _{k}}n_{\alpha
_{k}},\;n=\sum_{k=1}^{s}n_{\alpha _{k}},\text{ }\overline{\delta }=\frac{%
\sum_{i=1}^{s}\delta _{\alpha _{i}}n_{\alpha _{i}}}{n}.
\end{align*}%
We stress that $\delta _{ij}=\left\{ 
\begin{array}{c}
1\text{ if }i=j \\ 
0\text{ if }i\neq j%
\end{array}%
\right. $ is the Kronecker delta, and not connected to the the numbers $%
\delta _{\alpha _{i}}$ of internal degrees of freedom. This can be obtained
by observing that, cf. $\left( \ref{m4a}\right) $,%
\begin{equation*}
\frac{2}{T^{2}}\frac{Q_{\alpha _{i}}\int_{0}^{\infty }\varphi _{\alpha
_{i}}(I)I^{2}e^{-I/T}\,dI-\left( \int_{0}^{\infty }\varphi _{\alpha
_{i}}(I)Ie^{-I/T}\,dI\right) ^{2}}{Q_{\alpha _{i}}^{2}}=\delta _{\alpha _{i}}%
\text{, }i\in \left\{ 1,...,s\right\} \text{.}
\end{equation*}%
Then the degenerate values of $u$ become 
\begin{equation}
u_{0}=0\text{ and }u_{\pm }=\pm \sqrt{\frac{nT}{\rho }}\sqrt{\frac{d+2+%
\overline{\delta }}{d+\overline{\delta }}}\text{,}  \label{m6}
\end{equation}%
and the values of the signature $\left( k^{+},k^{-},l\right) $ of the
restriction of the quadratic form $\left( \left. B\phi \right\vert \phi
\right) $ to the kernel of $\mathcal{L}$ depending on\ the parameter $u$ are
given by table $\left( \ref{t2}\right) $, with $u_{\pm }$ in $\left( \ref{m6}%
\right) $.

\begin{acknowledgement}
The author is indebted to Prof. F. Golse for valuable discussions and kind
hospitality during visits to Paris and acknowledges the support by French
Institute in Sweden (through the FR\"{O} program in 2016) and SveFUM (in
2017 and 2019) for his visits.
\end{acknowledgement}

\bibliographystyle{siamproc}
\bibliography{biblo1}

\end{document}